\documentclass[12pt,a4paper,fleqn]{amsart}
\usepackage[utf8]{inputenc}
\usepackage[foot]{amsaddr}
\usepackage[pdftex]{lscape}
\usepackage{comment}
\usepackage[in]{fullpage}
\usepackage[colorlinks,breaklinks,linkcolor=blue,anchorcolor=blue,citecolor=blue]{hyperref}
\usepackage{amsfonts,amssymb,amsthm}
\usepackage[alphabetic,nobysame]{amsrefs}
\input amsart-modif
\numberwithin{equation}{section}\swapnumbers

\usepackage{parskip}

\usepackage[cmtip,matrix,arrow]{xy} \SelectTips{cm}{10}
\newcommand{\cxymatrix}[1]{\vcenter{\xymatrix@=15pt{#1}}}
\usepackage{rotating}

\newcommand{\xysubseteq}{\ar@{}[r]|{\displaystyle\subseteq}}
\newcommand{\xysubseteqdown}{\ar@{}[d]|{\rotatebox{90}{$\supseteq$}}}

\usepackage{tikz}
\usetikzlibrary{calc}

\usepackage{aliascnt}

\newtheorem{theorem}{Theorem}[section]
\newaliascnt{lemma}{theorem}

\newtheorem{lemma}[lemma]{Lemma}
\aliascntresetthe{lemma}
\newaliascnt{corollary}{theorem}

\newtheorem{corollary}[corollary]{Corollary}
\aliascntresetthe{corollary}
\newaliascnt{proposition}{theorem}

\newtheorem{proposition}[proposition]{Proposition}
\aliascntresetthe{proposition}

\theoremstyle{definition}
\newaliascnt{definition}{theorem}

\newtheorem{definition}[definition]{Definition}
\aliascntresetthe{definition}

\newtheorem*{ack}{Acknowledgment}

\newaliascnt{remark}{theorem}
\newtheorem{remark}[remark]{Remark}
\aliascntresetthe{remark}
\newtheorem{remarks}[remark]{Remarks}

\newtheorem*{remark*}{Remark}

\newaliascnt{example}{theorem}
\newtheorem{example}[example]{Example}

\aliascntresetthe{example}

\usepackage{enumitem}
\setlist[enumerate,2]{label=\textit{\alph*)},ref=\textit{\alph*})}
\setlist[enumerate,1]{label=\textit{\roman*)},ref=\textit{\roman*})}

\usepackage[capitalize]{cleveref}
\usepackage{microtype}

\def\ssy.{homogeneous spherical datum}
\def\ssys.{homogeneous spherical data}
\def\wss.{weak spherical datum}
\def\wsss.{weak spherical data}
\def\Wss.{Weak spherical datum}
\def\Wsss.{Weak spherical data}

\newcommand{\fg}{\mathfrak{g}}
\newcommand{\fh}{\mathfrak{h}}
\newcommand{\fl}{\mathfrak{l}}
\newcommand{\fs}{\mathfrak{s}}
\newcommand{\fo}{\mathfrak{o}}
\newcommand{\fp}{\mathfrak{p}}
\newcommand{\ft}{\mathfrak{t}}
\newcommand{\fa}{\mathfrak{a}}

\newcommand{\gl}{\fg\fl}
\renewcommand{\sl}{\fs\fl}
\newcommand{\so}{\fs\fo}
\renewcommand{\sp}{\fs\fp}

\newcommand{\Sq}{{\overline S}}
\newcommand{\cS}{{\mathcal S}}

\newcommand{\CC}{\mathbb{C}}
\newcommand{\QQ}{\mathbb{Q}}
\newcommand{\RR}{\mathbb{R}}
\newcommand{\ZZ}{\mathbb{Z}}

\newcommand{\leer}{\varnothing}
\renewcommand{\rho}{\varrho}
\renewcommand{\phi}{\varphi}
\renewcommand{\epsilon}{\varepsilon}

\newcommand{\<}{\langle} 
\renewcommand{\>}{\rangle}
\newcommand{\into}{\hookrightarrow}
\renewcommand{\[}{\begin{equation}}
\renewcommand{\]}{\end{equation}}
\let\mod\undefined
\DeclareMathOperator{\mod}{/\!\!/}
\DeclareMathOperator{\rk}{rk}
\DeclareMathOperator{\Span}{span}
\DeclareMathOperator{\Hom}{Hom}
\DeclareMathOperator{\Spec}{Spec}
\DeclareMathOperator{\Ad}{Ad}
\DeclareMathOperator{\res}{res}

\newcommand{\Lq}{{\overline L}}
\newcommand{\G}{{\mathbf{G}}}
\newcommand{\cD}{\mathcal{D}}

\newcommand{\cU}{\mathcal{U}}
\newcommand{\cC}{\mathcal{C}}
\newcommand{\cZ}{\mathcal{Z}}

\newcommand{\GL}{\mathrm{GL}}
\newcommand{\PGL}{\mathrm{PGL}}
\newcommand{\PSp}{\mathrm{PSp}}
\newcommand{\PSO}{\mathrm{PSO}}
\newcommand{\SL}{\mathrm{SL}}
\newcommand{\SO}{\mathrm{SO}}
\newcommand{\Sp}{\mathrm{Sp}}

\newcommand{\sA}{\mathsf{A}}
\newcommand{\sB}{\mathsf{B}}
\newcommand{\sC}{\mathsf{C}}
\newcommand{\sD}{\mathsf{D}}
\newcommand{\sE}{\mathsf{E}}
\newcommand{\sF}{\mathsf{F}}
\newcommand{\sG}{\mathsf{G}}

\renewcommand{\a}{\alpha}
\renewcommand{\b}{\beta}
\renewcommand{\d}{\delta}
\newcommand\e{\varepsilon}
\newcommand\g{\gamma}
\newcommand{\av}{\a^\vee}
\newcommand{\avq}{\overline\a^\vee}
\newcommand{\bv}{\b^\vee}
\newcommand{\bvq}{\overline\b^\vee}
\newcommand{\dv}{\d^\vee}
\newcommand{\gv}{\g^\vee}

\newcommand{\Gv}{G^\vee}
\newcommand{\Lv}{L^\vee}
\newcommand{\Tv}{T^\vee}
\newcommand{\Sv}{S^\vee}
\newcommand{\Av}{A^\vee}
\newcommand{\Xiv}{\Xi^\vee}
\newcommand{\Xiqv}{\Xiq^\vee}
\newcommand{\chiv}{\chi^\vee}
\newcommand{\tauv}{\tau^\vee}
\newcommand{\Phiv}{\Phi^\vee}
\newcommand{\fgv}{\fg^\vee}
\newcommand{\flv}{\fl^\vee}
\newcommand{\Ga}{G\ass}
\newcommand{\Ta}{T\ass}

\newcommand{\La}{L\ass}
\newcommand{\LaS}{\La_\cS}
\newcommand{\LvS}{\Lv_\cS}

\newcommand{\Siv}{\Sigma^\vee}
\newcommand{\siv}{\sigma^\vee}
\newcommand{\Sa}{\Sigma\ass}
\newcommand{\s}{{}^s}
\newcommand{\sa}{\sigma\ass}

\newcommand{\ass}{^\wedge}
\newcommand{\ad}{_{\rm ad}}
\newcommand{\p}{^p}
\newcommand{\+}{+\ldots+}
\newcommand{\Xiq}{\Lambda}
\newcommand{\cSq}{\mathcal R}
\renewcommand{\tilde}{\widetilde}
\newcommand{\half}{\textstyle{\frac12}}

\renewcommand{\L}{{}^{\mathsf L}}
\newcommand{\tGv}{\L G}

\def\8#1/{\underline{#1}}
\def\9#1,#2&{\hspace{-10pt}\left\{\hspace{-4pt}\begin{array}{l}#1\\#2\end{array}\right.&}
\def\7#1,#2&{\begin{array}{l}#1\\#2\end{array}&}
\def\5#1/{\hbox to 20pt{\hss$#1$\hss}}

\def\scalar{0.7}

\def\op(#1,#2){\draw (#1,#2) circle (3pt);}
\def\xp(#1,#2){\filldraw (#1,#2) circle (3pt);}

\def\oo(#1,#2;#3){\draw (#1,#2) node[above]{$\ph#3$} circle (3pt);}
\def\yy(#1,#2;#3){\filldraw (#1,#2) node[below]{$\ph#3$} circle (3pt);}
\def\ph{\vrule width 0pt height 8pt}
\def\xx(#1,#2;#3){\filldraw (#1,#2) node[above]{$#3$} circle (3pt);}
\def\ddd(#1,#2){%
\draw($(#1,#2)+(0.08,0)$) -- ($(#1,#2)+(0.6,0)$);
\filldraw ($(#1,#2)+(0.8,0)$) circle (0.3pt);
\filldraw ($(#1,#2)+(1,0)$) circle (0.3pt);
\filldraw ($(#1,#2)+(1.2,0)$) circle (0.3pt);
\draw($(#1,#2)+(1.4,0)$) -- ($(#1,#2)+(1.9,0)$);}

\title[]{The dual group of a spherical variety}

\author[]{Friedrich Knop}
\author[]{Barbara Schalke}
\address[]{Dept. Mathematik\\FAU Erlangen-Nürnberg\\
Cauerstraße 11\\
D-91058 Erlangen}


\begin{document}

\begin{abstract}

  Let $X$ be a spherical variety for a connected reductive group
  $G$. Work of Gaitsgory-Nadler strongly suggests that the Langlands
  dual group $\Gv$ of $G$ has a subgroup whose Weyl group is the
  little Weyl group of $X$. Sakellaridis-Venkatesh defined a refined
  dual group $\Gv_X$ and verified in many cases that there exists an
  isogeny $\phi$ from $\Gv_X$ to $\Gv$. In this paper, we establish
  the existence of $\phi$ in full generality. Our approach is purely
  combinatorial and works (despite the title) for arbitrary
  $G$-varieties.

\end{abstract}

\maketitle

\section{Introduction}

Let $G$ be a connected reductive group defined over an algebraically
closed field $k$ of characteristic zero. It is known for a while that
the large scale geometry of a $G$-variety $X$ is controlled by a root
system $\Phi_X$ attached to it. For spherical varieties (i.e.,
varieties where a Borel subgroup of $G$ has an open orbit) this was
observed by Brion \cite{Brion}. For the general case see
\cite{KnopAB}.

Root systems classify reductive groups. So it is tempting to ask
whether the group $G_X$ with root system $\Phi_X$ has any geometric
significance. In particular, it would be desirable to have a natural
homomorphism from $G_X$ to $G$. Unfortunately, examples show that this
is not possible. The simplest is probably $X=G/H$ where $G=\Sp(4,\CC)$
and $H=\G_m\times\Sp(2,\CC)$. Here $\Phi_X$ consists of the
\emph{short} roots of $G$, hence does not correspond to a subgroup of
$G$.

For spherical varieties a solution to this problem was proposed by
Gaitsgory-Nadler in \cite{GaitsgoryNadler}: instead of finding a map
from $G_X$ to $G$, one should look at the Langlands dual groups and
try to find a homomorphism $\Gv_X\to\Gv$ between them. In fact, using
the Tannakian formalism, they were able to construct a subgroup
$\Gv_{X,GN}$ of $\Gv$ which seems to have the right properties but the
fact that the root system of $\Gv_{X,GN}$ is $\Phiv_X$ remains
conjectural.

Later, Sakellaridis and Venkatesh, \cite{SV}, refined the notion of
the dual group $\Gv_X$ and used a hypothetical homomorphism
$\phi:\Gv_X\to\Gv$ to formulate a Plancherel theorem for spherical
varieties over $p$-adic fields. The homomorphism $\phi$ was described
in terms of what they call \emph{associated roots}. They proved the
uniqueness of $\phi$ in general and its existence in many cases.

The purpose of the present paper is to prove the existence of
$\phi:\Gv_X\to\Gv$ (in the sense of \cite{SV}) in full generality
(\cref{thm:main}). Our approach is completely combinatorial. More
precisely, we use a classification of rank-$1$ spherical varieties due
to Akhiezer \cite{Akhiezer} and, to a certain extent, the
classification of the rank-$2$ varieties by Wasserman \cite{Wasserman}
(verified by Bravi \cite{Bravi}).

Towards proving the existence of $\phi$, we show that the associated
roots of Sakellaridis-Venkatesh are the simple roots of a subgroup
$\Ga_X\subseteq\Gv$ (the associated group of $X$, \cref{thm:assexist})
and that $\phi$ should map $\Gv_X$ to $\Ga_X$. Observe that, by
construction, $\Ga_X$ is of maximal rank, i.e., it contains the
maximal torus $\Tv$ of $\Gv$. Subgroups of this type have been
classified by Borel-de Siebenthal in \cite{BdS}.

Now we show that $\phi:\Gv_X\to\Ga_X$ can be obtained by a process
which we call \emph{folding}. This is a slight generalization of the
usual folding by a graph automorphism.

It is curious that Ressayre, \cite{Res}, arrived at the same folding
procedure in his classification of minimal rank spherical
varieties. This means, in particular, that the homogeneous variety
$\Ga_X/\phi(\Gv_X)$ is of an extremely special type, namely it is
affine, spherical, and of minimal rank (\cref{cor:affsphminrk}).

Next we give, in the spirit of the Langlands philosophy, a
reformulation of the main results of \cite{KnopAB} (on moment maps)
and \cite{KnopHC} (on invariant differential operators) in terms of
the dual group (Section \S\ref{sec:moment}).

The theory of Sakellaridis-Venkatesh also calls for a particular
homomorphism $\SL(2)\to\Gv$ whose image centralizes $\phi(\Gv_X)$ and
whose existence we prove as well (\cref{cor:arthur}). To this end, we
determine the centralizer of $\phi(\Gv_X)$ in $\Gv$. More precisely,
we show (\cref{thm:centralizer}) that $\phi(\Gv_X)$ is centralized by
a finite index subgroup $\La_X$ of a fixed point group $(\Lv_X)^{W_X}$
where $\Lv_X\subseteq\Gv$ is a Levi subgroup and $W_X$ is the Weyl
group of $\Gv_X$ acting on $\Lv_X$ in a not quite obvious way. Under
some non-degeneracy condition we show (\cref{thm:fullcentralizer})
that $\La_X$ is even the entire centralizer of $\phi(\Gv_X)$ in $\Gv$.

All in all, we obtain the following poset of subgroups of $\Gv$:
\begin{center}
  $\cxymatrix{ &&\5\Gv/\ar@{-}[dl]\ar@{-}[dr]\\
    &\5\Ga_X/\ar@{-}[dr]\ar@{-}[dl]&&\5\Lv_X/\ar@{-}[dr]\ar@{-}[dl]\\
    \5\Gv_X/&&\5\Tv/&&\5\La_X/\\}$
\end{center}
In \cref{cor:GZC} we see that
\begin{equation}
  \Gv_X\times^{Z(\Gv_X)}\La_X\into\Gv
\end{equation}
is an injective homomorphism where $Z(\Gv_X)\subseteq\Gv_X$ is the
center.

To exemplify our results, we listed in \cref{tab:Kraemer} the Lie
algebras of all relevant subgroups for $X=G/H$ in Krämer's list
\cite{Kraemer}, i.e., for $G$ simple and $H$ reductive, spherical. One
case is particularly curious since it involves all six exceptional
groups (counting $\sD_4$):
\begin{equation}
  \fg=\fgv=\sE_8,\ \fh=\sE_7+\sl(2),\ \fg\ass_X=\sE_6+\ft^2,\
  \fgv_X=\sF_4,\ \flv_X=\sD_4+\ft^4,\ \fl\ass_X=\sG_2.
\end{equation}

In Section \S\ref{sec:L-groups} we study the behavior of the dual
group with respect to a (Galois) group $E$ of automorphisms giving
hopefully some hints on how to define an $L$-group $\L G_X$ of $X$. In
particular, we found that the action of $E$ on $\Gv_X$ will, in
general, not be the obvious one (i.e., the one induced by diagram
automorphisms).

In the final section, we discuss functoriality properties with respect
to various transformations of \wsss., like parabolic induction and
localization. We also note that the group $\Gv$, its subgroup $\Ga_X$,
the dual group $\Gv_X$, and the homomorphism $\phi$ are defined over
$\ZZ$. Moreover the centralizer is, in general, defined over
$\ZZ[\half]$.

The setting of the paper is actually more general than described
above. Instead of directly working with spherical varieties, we only
study them through an intermediate combinatorial structure which we
call a \emph{\wss.}. This structure is a weakening (whence the name)
of the \emph{\ssy.} of \cite{Luna} which is used to classify
homogeneous spherical varieties (by Bravi-Pezzini
\cite{BraviPezzini}). Additionally, one can associate a \wss. to
\emph{any} $G$-variety (\cref{prop:non-spherical}) which widens the
scope of our theory to this generality.

\begin{remark*}

  Parts of this paper are based on the second author's PhD-thesis
  \cite{Schalke}.

\end{remark*}

\begin{ack}

  Thanks are due to Yiannis Sakellaridis for helpful comments on
  $L$-groups and to Kay Paulus, Bart Van Steirteghem, and the referee
  for their care in reading preliminary versions of this article.

\end{ack}

\section{Notation}

If $\Xi$ is a lattice we denote its dual $\Hom(\Xi,\ZZ)$ by
$\Xiv$. The pairing between $\Xi$ and $\Xiv$ will be denoted by
$\<\cdot\mid\cdot\>$.

In the following let $(\Xiq,\Phi,\Xiqv,\Phiv)$ be a finite root datum
and let $S\subseteq\Phi$ be a fixed basis, i.e., a set of simple
roots. The quadruple $\cSq:=(\Xiq,S,\Xiqv,\Sv)$ will then be called a
\emph{based root datum}.

Let $\Phi^+\subseteq\Phi$ be the set of positive roots with respect to
$S$. The Weyl group is denoted by $W$. We also fix a $W$-invariant
scalar product $(\cdot,\cdot)$ on $\Xiq\otimes\RR$. It will only serve
auxiliary purposes and will not be considered part of the structure.

For any algebraic group $H$ let $\Xi(H)$ be its character group. The
Lie algebra of a group $G$, $H$, $L$ etc. will be denoted by the
corresponding fraktur letter $\fg$, $\fh$, $\fl$, etc.

Let $G$ be a connected reductive group defined over an algebraically
closed field $k$ of characteristic $0$ whose based root datum is
$\cSq$ (with respect to a Borel subgroup $B\subseteq G$ and a maximal
torus $T\subseteq B$). The $1$-dimensional unipotent root subgroup
corresponding to the root $\a\in\Phi$ will be denoted by $G_\a$. On
the other hand, $G(\a)$ is the semisimple rank-$1$-group generated by
$G_\a$ and $G_{-\a}$.

Since the dual based datum $(\Xiqv,\Sv,\Xiq,S)$ is a based root datum
as well, it is the root datum of a unique connected reductive group
$\Gv$, the \emph{dual group} of $G$. In this paper we take $\Gv$ to be
defined over $\CC$ even though most constructions work over $\ZZ$ (see
\cref{thm:integral}). This means that $G$ and $\Gv$ are not
necessarily defined over the same field.

A choice of generators $e_\a\in\fg_\a$, $\a\in S$ is called a
\emph{pinning}. We fix a pinning for $\Gv$.

\section{Some basic facts concerning root systems}

In this section we collect a couple of well-known criteria for root
(sub)systems.

\begin{proposition}\label{prop:cryst}

  Let $\Sigma$ be a subset of an Euclidean vector space $V$. Assume
  that $\Sigma$ is contained in some open half-space and that
  $\<\sigma\mid\tauv\>=\frac{2(\sigma,\tau)}{(\tau,\tau)}\in\ZZ_{\le0}$
  for all $\sigma\ne\tau\in\Sigma$. Then $\Sigma$ is the basis of a
  finite root system.

\end{proposition}

\begin{proof}

  By Bourbaki (Chap. V, \S3.5, Lemme 3(ii)), the two conditions imply
  that $\Sigma$ is linearly independent. Without loss of generality we
  may assume that $\Sigma$ is a basis of $V$.

  Now consider the Cartan matrix
  $C_{\tau\sigma}:=\<\sigma\mid\tauv\>$.  It is symmetrizable and its
  symmetrization is positive definite. It follows from
  \cite{Kac}*{Prop.~4.9} that $\Sigma$ is a basis of a finite root
  system inside $\ZZ\Sigma\subseteq V$.
\end{proof}

Recall that a root subsystem $\Psi\subseteq\Phi$ is additively closed
if $\Psi=\Phi\cap\ZZ\Psi$.

\begin{lemma}\label{lemma:addclosed}

  Let $\Phi$ be a finite root system, $\Psi\subseteq\Phi$ a root
  subsystem and $\Sigma\subseteq\Psi$ a basis. Then $\Psi$ is
  additively closed in $\Phi$ if and an only if any two elements of
  $\Sigma$ generate an additively closed root subsystem.

\end{lemma}

\begin{proof}

  We have to show that if $\phi=\sum_{i=1}^N\psi_i\in\Phi$ with
  $\psi_i\in\Psi$ then $\phi\in\Psi$. First we claim that it suffices
  to consider the case $N=2$. Indeed, from
  $0<(\phi,\phi)=\sum_{i=1}^N(\phi,\psi_i)$ we see that there is an
  $i$ with $(\phi,\psi_i)>0$. Let $\phi_0:=\phi-\psi_i$. Then either
  $\phi_0=0$, in which case $\phi=\psi_i\in\Psi$, or
  $\phi_0\in\Phi$. Since then $\phi_0=\sum_{j\ne i}\psi_j\in\Psi$, by
  induction on $N$ we see that $\phi=\phi_0+\psi_i\in\Psi$ by the case
  $N=2$.

  So assume $\phi=\psi_1+\psi_2\in\Phi$ with $\psi_i\in\Psi$. Then
  $\Psi':=\Span_\QQ(\psi_1,\psi_2)\cap\Psi$ is an additively closed
  subsystem of $\Psi$. Hence every basis $\psi_1',\psi_2'\in\Psi'$ can
  be extended to a basis $\Sigma'\subseteq\Psi$ (just choose a linear
  function $\ell$ with $0<\ell(\psi_i')<1$ for $i=1,2$ and
  $|\ell(\psi)|>1$ for all $\psi\in\Psi\setminus\Psi'$ and consider
  the indecomposable elements of $\Psi\cap\{\ell>0\}$). Let $w$ be the
  element of the Weyl group of $\Psi$ such that
  $w\Sigma'=\Sigma$. Since $w\Psi'$ is additively closed in $\Phi$ by
  assumption, so is $\Psi'$. This implies $\phi\in\Psi$.
\end{proof}

\begin{lemma}\label{lemma:addclosed3}

  A subset $\Sigma\subseteq\Phi^+$ is the basis of an additively
  closed root subsystem if and only if $\sigma-\tau\not\in\Phi^+$ for
  all $\sigma,\tau\in\Sigma$.

\end{lemma}

\begin{proof}

  Clearly, the condition implies $\tau-\sigma\not\in\Phi^+$ and
  therefore $\sigma-\tau\not\in\Phi$ for all $\sigma\ne\tau\in\Sigma$.
  From this we infer $\<\sigma\mid\tauv\>\in\ZZ_{\le0}$. Also the
  half-space condition of \cref{prop:cryst} is satisfied since
  $\Sigma\subseteq\Phi^+$. Thus, $\Sigma$ is a basis of some root
  system $\Psi\subseteq\Phi$. For $\sigma,\tau\in\Sigma$ let
  $\Phi':=\Span_\QQ(\sigma,\tau)\cap\Phi$ and
  $\Psi':=\Span_\QQ(\sigma,\tau)\cap\Psi$. If $\Psi'$ were not
  additively closed in $\Phi'$, then $\Phi'$ would be of type $\sB_2$
  or $\sG_2$ and $\Psi'$ would be its subset of short roots. But then
  $\sigma-\tau\in\Phi'\subseteq\Phi$, contrary to our assumption. Now
  \cref{lemma:addclosed} implies that $\Psi$ is additively closed in
  $\Phi$.
\end{proof}

\section{Folding root systems}

The process of folding a based root system by a graph automorphism is
well-known (see e.g. \cite{Springer}*{\S10}). We are going to need a
slight generalization.

For this we start with the based root datum $(\Xiq,S,\Xiqv,\Sv)$ of
the connected group $G$. Let $\alpha\mapsto\s\a$ an involution on
$S$. With $\s\av:=\s(\av):=(\s\a)^\vee$, we get also an involution of
$\Sv$.

\begin{definition}\label{def:folding}

  The involution $s$ is called a \emph{folding} if for all
  $\a,\b\in S$:

  \begin{enumerate}

  \item\label{it:fold1} $\<\a\mid\s\av\>=0$ whenever $\a\ne\s\a$ and

  \item\label{it:fold2} $\<\a-\s\a\mid\bv+\s\bv\>=0$.

  \end{enumerate}

\end{definition}

Observe, that we do not assume $s$ to be an automorphism of the Dynkin
diagram $\cD$ of $G$. This would be equivalent to
\begin{equation}
  \<\s\a\mid\s\bv\>=\<\a\mid\bv\>
\end{equation}
which implies property \ref{it:fold2} of a folding.

\begin{example}

  Not all foldings are automorphisms, though. Let $G$ be of type
  $\sB_3$ with roots $\a_1,\a_2,\a_3$. Let $\s\a_1=\a_3$ and
  $\s\a_2=\a_2$. Then $s$ is a folding but not a diagram
  automorphism. Indeed, the only case which has to be verified for
  \ref{it:fold2} is $\a=\a_1$ and $\b=\a_2$. Then
  \begin{equation}
    \<\a-\s\a\mid\bv+\s\bv\>=\<\a_1-\a_3\mid 2\av_2\>=0
  \end{equation}
  shows that $s$ is a folding.

\end{example}

We show that this example is essentially the only folding that is not
a diagram automorphism.

\begin{lemma}\label{lemma:main}

  \begin{enumerate}

  \item\label{it:mainlemma1} If $\a\ne \s\a\in S$ then $\a$ is not
    connected to $\s\a$ in $\cD$.

  \item\label{it:mainlemma2} For $\a\ne\b\in\cD$ assume that the types
    of the edges between $\a,\b$ and between $\s\a,\s\b$ are
    different, i.e.,
    $ \<\a\mid\bv\>\ne\<\s\a\mid\s\bv\>\text{ or }
    \<\b\mid\av\>\ne\<\s\b\mid\s\av\>$. Then
    $S_0:=\{\a,\b,\s\a,\s\b\}$ spans a subdiagram of $\cD$ which is of
    type $\sB_3$.

  \end{enumerate}
\end{lemma}

\begin{proof}

  Assertion \ref{it:mainlemma1} is just the defining property
  \ref{it:fold1} of a folding. For \ref{it:mainlemma2} observe first
  that $\b=\s\a$ can't happen. So we may assume that the orbits
  $\{\a,\s\a\}$ and $\{\b,\s\b\}$ are disjoint.

  Let $\cD_0$ be the Dynkin diagram of $S_0$. Then, according to the
  number $f$ of $s$-fixed points in $S_0$, there are three cases to be
  distinguished:

  $f=2$: Here $s$ acts as identity on $S_0$ and \ref{it:mainlemma2} is
  trivially satisfied.

  $f=1$. Without loss of generality we may assume that $\a\ne\s\a$ and
  $\b=\s\b$. The underlying simply laced graph of $\cD_0$ looks like
  this:
  \begin{equation}
    \vcenter{\hbox{$
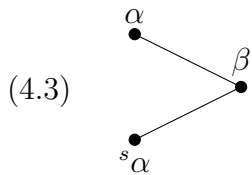
        \begin{tikzpicture}[scale=\scalar]
          \xx(0,2;\a) \xx(2,1;\b) \yy(0,0;\s\a) \draw(0,2) -- (2,1);
          \draw(0,0) -- (2,1);
        \end{tikzpicture}$}}
  \end{equation}
  The second folding property \ref{it:fold2} implies
  \begin{equation}
    a:=\<\a\mid\bv\>=\<\s\a\mid\bv\>.
  \end{equation}
  The case $a=0$ can't happen under the assumptions of
  \ref{it:mainlemma2}. If $a\le-2$ then $\cD_0$ would contain two
  arrows which is impossible. So assume $a=-1$. This means that $\b$
  cannot be shorter than $\a$ or $\s\a$ which leaves for $\cD_0$ only
  the types $\sA_3$ and $\sB_3$, confirming the assertion.

  $f=0$: We claim that $s$ acts as an automorphisms on $\cD_0$. Its
  underlying simply laced graph could a priori look like this:
  \begin{equation}
    \vcenter{\hbox{$
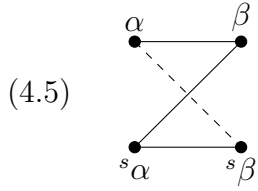
\begin{tikzpicture}[scale=\scalar] \xx(0,2;\a)
          \xx(2,2;\b) \yy(0,0;\s\a) \yy(2,0;\s\b) \draw (0,2)--(2,2);
          \draw[dashed] (0,2)--(2,0); \draw (0,0)--(2,2); \draw
          (0,0)--(2,0);
        \end{tikzpicture}$}}
  \end{equation}
  Since $\cD_0$ is not a cycle at least one of the diagonals is
  missing. Without loss of generality we may choose it to be the
  dashed line, i.e., we assume
  \begin{equation}
    \<\a\mid\s\bv\>=\<\s\b\mid\av\>=0.
  \end{equation}
  Now property \ref{it:fold2} implies
  \begin{equation}
    \<\a-\s\a\mid\bv+\s\bv\>=\<\b-\s\b\mid\av+\s\av\>=0.
  \end{equation}
  From this, we get the two equations
  \begin{equation}\label{eq:arrow}
    \begin{array}{rll}
      \<\a\mid\bv\>&=\ \<\s\a\mid\bv\>\ +&\<\s\a\mid\s\bv\>\\
      \<\b\mid\av\>&+\ \<\b\mid\s\av\>\ = &\<\s\b\mid\s\av\>
    \end{array}
  \end{equation}
  Observe that all numbers involved are non-positive. Hence, if
  $\<\a\mid\bv\>=0$ or $\<\s\b\mid \s\av\>=0$ then all other terms are
  $0$. Then $\cD_0$ is of type $4\sA_1$ and the assertion is true.

  Now assume that both $\<\a\mid\bv\>$ and $\<\s\b\mid\s\av\>$ are
  $\le-1$. If the middle terms (corresponding to the diagonal edge)
  were also $\le-1$ then both $\<\a\mid\bv\>$ and $\<\s\b\mid\s\av\>$
  were even $\le-2$. Since $\cD_0$ does not contain two arrows this is
  impossible. This forces $\<\a\mid\s\bv\>=\<\s\b\mid\av\>=0$, i.e.,
  the other diagonal edge is absent, too. But then \eqref{eq:arrow}
  boils down to $\a,\b$ and $\s\a,\s\b$ being connected by the same
  type of edge proving the assertion also in this case.
\end{proof}

Now it is easy to classify foldings.

\begin{corollary}\label{cor:folding}

  Every folding is a disjoint union of the following foldings:

  \begin{itemize}

  \item A component where $s$ acts trivially.

  \item Two isomorphic components which are interchanged by $s$.

  \item One of the following four cases:
    \begin{equation}\label{eq:diagrams}
      \vcenter{\hbox{$ 
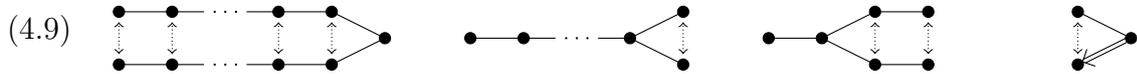
\begin{tikzpicture}[scale=\scalar]

            \xx(0,0;) \draw[densely dotted,<->] (0,0.2) -- (0,0.8);
            \draw (0.1,0)--(0.9,0); \xx(1,0;) \draw[densely
            dotted,<->] (1,0.2) -- (1,0.8); \ddd(1,0) \xx(3,0;)
            \draw[densely dotted,<->] (3,0.2) -- (3,0.8); \draw
            (3.1,0)--(3.9,0); \xx(4,0;) \draw[densely dotted,<->]
            (4,0.2) -- (4,0.8); \xx(0,1;) \draw (0.1,1)--(0.9,1);
            \xx(1,1;) \ddd(1,1) \xx(3,1;) \draw (3.1,1)--(3.9,1);
            \xx(4,1;) \draw (4,1)--(5,0.5); \xx(5,0.5;) \draw
            (4,0)--(5,0.5);

          \end{tikzpicture}
          \qquad
          \begin{tikzpicture}[scale=\scalar]

            \xx(0,0.5;) \draw (0.1,0.5)--(0.9,0.5); \xx(1,0.5;)
            \ddd(1,0.5) \xx(3,0.5;) \xx(4,1;) \xx(4,0;) \draw
            (3.0,0.5) -- (4,1); \draw (3.0,0.5) -- (4,0);
            \draw[densely dotted,<->] (4,0.2) -- (4,0.8);

\end{tikzpicture}
\qquad
\begin{tikzpicture}[scale=\scalar]

  \xx(0,0.5;) \xx(1,0.5;) \xx(2,1;) \xx(2,0;) \xx(3,1;) \xx(3,0;)
  \draw (0,0.5)--(1,0.5); \draw (1,0.5)--(2,1); \draw (1,0.5)--(2,0);
  \draw (2,1)--(3,1); \draw (2,0)--(3,0); \draw[densely dotted,<->]
  (2,0.2) -- (2,0.8); \draw[densely dotted,<->] (3,0.2) -- (3,0.8);

\end{tikzpicture}
\qquad\qquad
\begin{tikzpicture}[scale=\scalar]

  \xx(0,2;) \xx(1,1.5;) \xx(0,1;) \draw(0.1,1.95) -- (0.9,1.55);
  \draw(0.095,1.095) -- (0.905,1.5); \draw(0.14,1.03) -- (0.92,1.42);
  \draw(0.09,1.05) -- (0.17,1.29); \draw(0.09,1.05) -- (0.33,0.97);
  \draw[densely dotted,<->] (0,1.2) -- (0,1.8);

\end{tikzpicture}$}}
\end{equation}

\end{itemize}

\end{corollary}

Observe that in all cases but the last, $s$ is a graph automorphism.

For $\a\in S$ we define the orbit sum
\begin{equation}
  \avq:=\begin{cases}\av&\text{if
    }\a=\s\a\\\av+\s\av&\text{otherwise}.\end{cases}
\end{equation}
and $\Sq^\vee:=\{\avq\mid\a\in S\}$. Then property \ref{it:fold1} of a
folding implies $\<\a\mid\avq\>=2$ while property \ref{it:fold2} is
equivalent to $\<\a\mid\bvq\>=\<\s\a\mid\bvq\>$. We are going to show
that the sets $S/\<s\>$ and $\overline \Sq^\vee$ are the roots and the
coroots of a subgroup of $G$. More generally, we construct coverings
of such a subgroup.

To this end, let $\Xi$ be a lattice and let $r:\Xiq\to\Xi$ be a
homomorphism with finite cokernel. Then $r^\vee:\Xiv\to\Xiqv$ will be
injective which means, in particular, that $\Xi$ and $r^\vee(\Xiv)$
are still dual to each other. Let $A$ be the torus with
$\Xi(A)=\Xi$. Then $r$ induces a homomorphism $\phi_A:A\to T$ with
finite kernel.

\begin{lemma}\label{lemma:folded}
  
  Let $s$ be a folding of the root system of $G$. Assume that
  $r(\a-\s\a)=0$ for all $\a\in S$ and that
  $\Sq^\vee\subseteq r^\vee(\Xiv)$. Then there is a connected
  reductive group $H$ with based root datum
  $(\Xi,r(S),r^\vee(\Xiv),\Sq^\vee)$ and a homomorphism $\phi:H\to G$
  with finite kernel such that $\phi|_A=\phi_A$.

\end{lemma}

\begin{proof}
  We construct $H$ in three stages. First, we construct a subgroup
  $H\ad$ of the adjoint group $G\ad:=G/Z(G)$ having the root datum
  $(\Xi\ad,r(S),\Xi\ad^\vee,\Sq^\vee)$ where
  $\Xi\ad:=r(\ZZ\,S)=\ZZ\, r(S)$. To this end we may assume that the
  folding is one of the indecomposable types of \cref{cor:folding}. In
  the case $s$ is a graph automorphism the existence of $H$ is well
  known (see e.g. \cite{Springer}*{Prop.~10.3.5}): the choice of a
  pinning $e_\a\in\fg_\a$ extends the $s$-action to an action on
  $G\ad$ and $H\ad$ will be the connected fixed point group
  $(G\ad^s)^\circ$.
  
  If $s$ is of the last type we have to show that $G\ad=\SO(7)$ (the
  adjoint group of type $\sB_3$ with simple roots $\a_1,\a_2,\a_3$)
  contains a subgroup $H\ad$ of rank $2$ such that $\a_1$ and $\a_3$
  restrict to the same simple root of $H\ad$ and $\a_2$ restricts to
  the other. Of course, such a subgroup is well known, namely
  $H\ad=\sG_2$. To see this let $\a_s$ and $\a_l$ be the two simple
  roots of $\sG_2$ and consider its $7$-dimensional representation. It
  has seven weights, namely
  \begin{equation}
    2\a_s+\a_l,\a_s+\a_l,\a_s,0,-\a_s,-\a_s-\a_l,-2\a_s-\a_l
  \end{equation}
  It is known that the $\sG_2$-action preserves a non-degenerate
  quadratic form, yielding an embedding $\sG_2\into\SO(7)$. The simple
  roots of $\SO(7)$ restrict to $\sG_2$ as required:
  \begin{align}
    &\res\a_1=(2\a_s+\a_l)-(\a_s+\a_l)=\a_s\nonumber\\
    &\res\a_2=(\a_s+\a_l)-\a_s=\a_l\\
    &\res\a_3=\a_s\nonumber
  \end{align}
  This establishes the existence of $H\ad$ also in this case.

  Now let $p:G\to G\ad$ be the projection and $H_1=p^{-1}(H\ad)^\circ$
  the connected preimage. Then the based root datum of $H_1$ is
  $(\Xi_1,r_1(S),\Xiv_1,\Sq^\vee)$ where
  \begin{align}
    &\Xiv_1=\{\chiv\in\Xiqv\mid
      \<\a\mid\chiv\>=\<\s\a\mid\chiv\>\text{ for all $\a\in S$}\}\text{ and }\\
    &\Xi_1=\Xiq/K\text{ with }K=\Span_\QQ(\a-\s\a\mid\a\in S)\cap\Xiq.
  \end{align}
  Moreover, $r_1:\Lambda\to\Xi_1$ is the projection. The conditions on
  $\Xi$ and $r$ ensure that $r$ factors through $\Xi_1$ and that
  $\Xiv$ contains the coroots inside $\Xiv_1$. The isogeny theorem
  \cite{Springer}*{9.6.5} then shows that there is an isogeny
  $H\to H_1$ inducing $\phi_A$ on $A\subseteq H$.
\end{proof}

\begin{corollary}\label{cor:minrank}

  Let $\phi:H\to G$ be obtained by folding as in
  \cref{lemma:folded}. Then:

  \begin{enumerate}

  \item\label{it:minrank1} Let $Z(G)$ be the center of a group
    $G$. Then $Z(H)=\phi^{-1}(Z(G))$. In particular, $\phi$ induces an
    injective homomorphism $H\ad\into G\ad$ between adjoint groups.

  \item\label{it:minrank2} The variety $G\ad/H\ad$ is a product of
    factors isomorphic to one of:
    \begin{equation}\label{eq:minranklist}.
      \begin{array}{ll}
        \bullet\quad K/K&\text{$K$ simple, adjoint,}\\
        \bullet\quad(K\times K)/{\rm diag}\,K&\text{$K$ simple, adjoint,}\\
        \bullet\quad\PGL(2n)/\PSp(2n)&n\ge 2,\\
        \bullet\quad\PSO(2n)/\SO(2n-1)&n\ge 4,\\
        \bullet\quad\sE_6^{\rm ad}/\sF_4,\\
        \bullet\quad\SO(7)/\sG_2.
      \end{array}
    \end{equation}

  \end{enumerate}

\end{corollary}

\begin{proof}

  \cref{it:minrank1} follows from the fact that the center of a
  reductive group is the common kernel of its simple roots and that
  the simple roots of $H$ are the restrictions of the simple roots of
  $G$. Now the items of \ref{it:minrank2} correspond precisely to
  those of \cref{cor:folding}.
\end{proof}

The very same list of diagrams as in \cref{cor:folding} already
appeared in a different context. For this let $X=G/H$ be a homogeneous
spherical variety. Attached to it is a lattice $\Xi(X)$ (see the
paragraph before \cref{prop:non-spherical} below for a
definition). Its rank is called the \emph{rank $\rk X$ of $X$}. It is
easy to see that the ranks of $G$, $H$, and $G/H$ satisfy the
inequality
\begin{equation}\label{eq:rank}
  \rk G/H\ge\rk G-\rk H
\end{equation}
Spherical varieties for which \eqref{eq:rank} is an equality are
called \emph{of minimal rank} and have been classified by Ressayre in
\cite{Res}. The point is now that when $G/H$ is affine (i.e., when $H$
is reductive) Ressayre obtains the same list as above. Thus we obtain:

\begin{corollary}\label{cor:affsphminrk}

  Let $\phi:H\to G$ be as in \cref{lemma:folded}. Then $G/\phi(H)$ is
  an affine spherical variety of minimal rank.

\end{corollary}

It is recommended to consult \cite{Res} for further properties of
minimal rank varieties.

\section{\Wsss.}\label{sec:WSS}

A $G$-variety $X$ is called \emph{spherical} if $B$ has an open orbit
in $G$. Homogeneous spherical varieties have been classified by Luna
and Bravi-Pezzini, \cites{Luna,BraviPezzini}, in terms of a
combinatorial structure called a \emph{\ssy.}. In addition to the
based root datum of $G$, it consists of a quintuple
$(\Xi,\Sigma,\cD,c,M)$ where $\Xi$ is a subgroup of $\Xiq$ (the weight
lattice), $\Sigma$ is a finite subset of $\Xi$ (the spherical roots),
$\cD$ is a finite set (the colors), $c$ is a map $\cD\to\Xiv$, and $M$
is a subset of $\cD\times S$. These objects are subject to a number of
axioms (see e.g. \cite{Luna}*{\S2.2}) which we won't repeat.

In practice, it is useful to work with a structure which contains less
information than a \ssy.. It is obtained by discarding most
information on $\cD$ and renormalizing the elements of $\Sigma$. There
are at least two reasons for doing so: first, these weaker structures
are much easier to handle while retaining most essential information
of a \ssy.. Secondly, unlike \ssys., it is possible to assign this
weaker structure to any $G$-variety (spherical or not, see
\cref{prop:non-spherical}). So they have a much wider scope.

\begin{definition}\label{def:WSS}

  A \emph{\wss. (with respect to $G$ or $\cSq$)} is a triple
  $(\Xi,\Sigma,S\p)$ where $\Xi\subseteq\Xiq$ is a subgroup and
  $\Sigma\subseteq\Xi$, $S\p\subseteq S$ are subsets such that the
  following axioms are satisfied:

  \begin{enumerate}

  \item\label{it:wss1} For every $\sigma\in\Sigma$ there is a subset
    $|\sigma|\subseteq S$ (its \emph{support}) such that
    $\sigma=\sum_{\a\in|\sigma|}n_\a\a$ with $n_\alpha\ge1$ and such
    that the triple $(|\sigma|,n_*,|\sigma|\cap S\p)$ appears in
    \cref{tab:NSR} (where the $n_\a$'s are the labels and the elements
    of $S\p$ are the black vertices).

  \item\label{it:wss2} Let $\a\in S\p$. Then $\<\Xi\mid\av\>=0$.

  \item\label{it:wss3} Let $\sigma=\a+\b\in\Sigma$ be of type $\sD_2$,
    i.e., $\a,\b\in S$ with $\a\perp\b$. Then $\<\Xi\mid\av-\bv\>=0$.

  \item\label{it:wss4} Let $\a,\b\in S$ with $\a,\a+\b\in\Sigma$. Then
    $\<\b\mid\av\>\ne-1$.

  \end{enumerate}

\end{definition}

\begin{table}[!h]

  \begin{equation}
    \begin{array}{llll}

      |\sigma|&\sigma\text{ and }|\sigma|\cap S\p\\

      \noalign{\smallskip\hrule\smallskip}

      \sA_1&
             \begin{tikzpicture}[scale=\scalar]
               \oo(0,0;1)
             \end{tikzpicture}
      \\

      \sA_n,\ n\ge2
              &
                \begin{tikzpicture}[scale=\scalar]
                  \oo(0,0;1) \draw (0.1,0)--(0.9,0); \xx(1,0;1) \draw
                  (1.1,0)--(1.9,0); \xx(2,0;1) \ddd(2,0) \xx(4,0;1)
                  \draw (4.1,0)--(4.9,0); \oo(5,0;1)
                \end{tikzpicture}

      \\
      \sB_n,\ n\ge2&
                     \begin{tikzpicture}[scale=\scalar]
                       \oo(0,0;1) \draw (0.1,0)--(0.9,0); \xx(1,0;1)
                       \draw (1.1,0)--(1.9,0); \xx(2,0;1) \ddd(2,0)
                       \xx(4,0;1) \draw (4.1,0.04)--(4.85,0.04); \draw
                       (4.1,-0.04)--(4.85,-0.04); \xx(5,0;1) \draw
                       (4.9,0) -- (4.7,0.15); \draw (4.9,0) --
                       (4.7,-0.15);
                     \end{tikzpicture}
      \\

      \sB_n,\ n\ge2&
                     \begin{tikzpicture}[scale=\scalar]
                       \oo(0,0;1) \draw (0.1,0)--(0.9,0); \xx(1,0;1)
                       \draw (1.1,0)--(1.9,0); \xx(2,0;1) \ddd(2,0)
                       \xx(4,0;1) \draw (4.1,0.04)--(4.85,0.04); \draw
                       (4.1,-0.04)--(4.85,-0.04); \oo(5,0;1) \draw
                       (4.9,0) -- (4.7,0.15); \draw (4.9,0) --
                       (4.7,-0.15);
                     \end{tikzpicture}
      \\

      \sC_n,\ n\ge3&
                     \begin{tikzpicture}[scale=\scalar]
                       \xx(0,0;1) \draw (0.1,0)--(0.9,0); \oo(1,0;2)
                       \draw (1.1,0)--(1.9,0); \xx(2,0;2) \ddd(2,0)
                       \xx(4,0;2) \draw (4.15,0.04)--(4.9,0.04); \draw
                       (4.15,-0.04)--(4.9,-0.04); \xx(5,0;1) \draw
                       (4.1,0) -- (4.3,0.15); \draw (4.1,0) --
                       (4.3,-0.15);
                     \end{tikzpicture}
      \\

      \sC_n,\ n\ge3&
                     \begin{tikzpicture}[scale=\scalar]
                       \oo(0,0;1) \draw (0.1,0)--(0.9,0); \oo(1,0;2)
                       \draw (1.1,0)--(1.9,0); \xx(2,0;2) \ddd(2,0)
                       \xx(4,0;2) \draw (4.15,0.04)--(4.9,0.04); \draw
                       (4.15,-0.04)--(4.9,-0.04); \xx(5,0;1) \draw
                       (4.1,0) -- (4.3,0.15); \draw (4.1,0) --
                       (4.3,-0.15);
                     \end{tikzpicture}
      \\

      \sF_4&
             \begin{tikzpicture}[scale=\scalar]
               \xx(0,0;1) \draw (0.1,0)--(0.9,0); \xx(1,0;2) \draw
               (1.1,0.04)--(1.85,0.04); \draw
               (1.1,-0.04)--(1.85,-0.04); \xx(2,0;3) \draw (1.9,0) --
               (1.7,0.15); \draw (1.9,0) -- (1.7,-0.15); \oo(3,0;2)
               \draw (2.1,0)--(2.9,0);
             \end{tikzpicture}
      \\

      \sG_2&
             \begin{tikzpicture}[scale=\scalar]
               \oo(4,0;2) \draw (4.17,0.06)--(4.9,0.06); \draw
               (4.1,0)--(4.9,0); \draw (4.17,-0.06)--(4.9,-0.06);
               \xx(5,0;1) \draw (4.1,0) -- (4.3,0.15); \draw (4.1,0)
               -- (4.3,-0.15);
             \end{tikzpicture}
      \\

      \sG_2&
             \begin{tikzpicture}[scale=\scalar]
               \oo(0,0;1) \draw (0.17,0.06)--(0.9,0.06); \draw
               (0.1,0)--(0.9,0); \draw (0.17,-0.06)--(0.9,-0.06);
               \oo(1,0;1) \draw (0.1,0) -- (0.3,0.15); \draw (0.1,0)
               -- (0.3,-0.15);
             \end{tikzpicture}
      \\

      \noalign{\smallskip\hrule\smallskip}

      \sD_2&
             \begin{tikzpicture}[scale=\scalar]
               \oo(0,0;1) \oo(1,0;1)
             \end{tikzpicture}
      \\

      \sD_n,\ n\ge3&
                     \vcenter{\hbox{$\begin{tikzpicture}[scale=\scalar]
                           \oo(0,0;2)
                           \draw (0.1,0)--(0.9,0);
                           \xx(1,0;2)
                           \ddd(1,0)
                           \xx(3,0;2)
                           \xx(3.5,0.5;1)
                           \xx(3.5,-0.5;1)
                           \draw (3.0,0) -- (3.5,0.5);
                           \draw (3.0,0) -- (3.5,-0.5);
                         \end{tikzpicture}$}}
      \\

      \sB_3&
             \begin{tikzpicture}[scale=\scalar]
               \xx(0,0;1) \draw (0.1,0)--(0.9,0); \xx(1,0;2) \draw
               (1.1,0.04)--(1.85,0.04); \draw
               (1.1,-0.04)--(1.85,-0.04); \oo(2,0;3) \draw (1.9,0) --
               (1.7,0.15); \draw (1.9,0) -- (1.7,-0.15);
             \end{tikzpicture}
      \\

      \noalign{\smallskip\hrule\smallskip}

    \end{array}
  \end{equation}
  \caption{The weak spherical roots}
  \label{tab:NSR}
\end{table}

\cref{tab:NSR} is derived from Akhiezer's classification
\cite{Akhiezer} of spherical varieties of rank $1$. More precisely,
that classification yields a list of all $\sigma$ which can be an
element of $\Sigma$ for some \ssy.. That list is, e.g., reproduced in
\cite{KnopSRSV}. An inspection shows that it contains entries which
are multiples of each other. Then \cref{tab:NSR} is obtained by only
picking those spherical roots which are primitive in the root lattice
of $G$. More precisely, for every spherical root $\sigma$ there is a
unique factor $c\in\{\frac12,1,2\}$ such that
$\sigma_{\rm norm}:=c\sigma$ is a weak spherical root.

\begin{remark}

  The normalization of spherical roots through primitive elements in
  the root lattice was proposed in \cite{SV}.

\end{remark}

This generalizes: let $\tilde\cS=(\tilde\Xi,\tilde\Sigma,\cD,c,M)$ be
a \ssy.. Then it is easy to deduce from the axioms satisfied by
$\tilde\cS$ that $\cS=(\Xi,\Sigma,S\p)$ is a \wss. where
\begin{equation}
  \begin{array}{lll}
    \Sigma&={}&\{\sigma_{\rm norm}\mid\sigma\in\tilde\Sigma\}\\
    \Xi&={}&\tilde\Xi+\ZZ\Sigma,\\
    S\p&={}&\{\a\in S\mid\text{there is no $D\in\cD$ with }(D,\a)\in M\}.
  \end{array}
\end{equation}

\begin{remark}
  Observe that $\tilde\Xi\subseteq\Xi\subseteq\half\tilde\Xi$ which
  means that the quotient $\Xi/\tilde\Xi$ is isomorphic to
  $(\ZZ/2\ZZ)^m$ for some $m\ge0$. Note that $m\le\rk G$. The upper
  bound is reached when $G$ is of adjoint type and $X=G/H$ is a
  symmetric variety of the same rank as $G$ (here $H$ is the full
  fixed point subgroup of an involution). Then $\tilde\Sigma=2S$ and
  $\tilde\Xi=\ZZ\tilde\Sigma$. Thus $\Sigma=S$ and therefore
  $\Xi/\tilde\Xi=\ZZ S/\ZZ(2S)\cong(\ZZ/2\ZZ)^{\rk G}$.

\end{remark}

By way of passing from $\tilde\cS$ to $\cS$ one loses not only the
information on the multiplier $c$ but also all information on $\cD$
except for which simple roots occur in $M$. On the other hand, it is
possible to define a \wss. for an arbitrary possibly non-spherical
$G$-variety $X$. More precisely, $\tilde\Xi(X)$, the set of characters
$\chi_f$ where $f$ is a $B$-semiinvariant rational function on $X$,
makes sense for any $X$. Furthermore, there are several ways to attach
a set $\tilde\Sigma(X)$ of spherical roots to $X$ (e.g., two of them
in \cite{KnopAuto}*{\S6}) which all differ just by the length of their
roots. So the set $\Sigma$ of normalized roots will be well
defined. To characterize $S\p$ let $P_\a\subseteq G$ be the minimal
parabolic attached to the simple root $\a\in S$. Then:

\begin{proposition}\label{prop:non-spherical}

  Let $X$ be a $G$-variety. Then $(\Xi,\Sigma,S\p)$ is a \wss. where
  \begin{equation}\label{eq:non-spherical}
    \begin{array}{lll}
      \Sigma&:=&\{\sigma_{\rm norm}\mid\sigma\in\Sigma(X)\}\\
      \Xi&:=&\Xi(X)+\ZZ\Sigma,\\
      S\p&:=&\{\a\in S\mid P_\a x=Bx\text{ for $x$ in a dense subset of $X$}\}.
    \end{array}
  \end{equation}
\end{proposition}

\begin{proof}

  The assertion is well-known but the proof is somewhat scattered in
  the literature. Let $W(X)$ be the Weyl group of the root system
  generated by $\Sigma(X)$ and let $\cC$ be its dominant Weyl
  chamber. Observe that $\Sigma$ can be recovered from $\cC$ alone.
  From \cite{KnopAB}*{Thm.~7.4} it is known that $-\cC$ coincides with
  the so-called central valuation cone $\cZ(X)$. If $X$ is not yet
  spherical then, using \cite{KnopIB}*{Kor.~9.5, Satz~7.5}, one can
  show that there is a variety $X'$ whose complexity is decreased by
  one but with $\Xi(X')\otimes\QQ=\Xi(X)\otimes\QQ$ and
  $\cZ(X')=\cZ(X)$. Moreover $S\p(X')=S\p(X)$ by
  \cite{KnopAB}*{\S2}. This reduces the assertion to spherical
  varieties where it is known.
\end{proof}

We return to abstract \wsss.. The following important property is not
at all apparent from the definition:

\begin{lemma}\label{lemma:le0}

  Let $(\Xi,\Sigma,S\p)$ be a \wss. and let $\sigma,\tau\in\Sigma$ be
  distinct. Then $(\sigma,\tau)\le0$.

\end{lemma}

\begin{proof}

  Lemma~5.2 of \cite{KnopSRSV} lists all possible pairs
  $\sigma\ne\tau\in\Sigma$ with connected support, $(\sigma,\tau)>0$
  and satisfying axioms \ref{it:wss1}, \ref{it:wss2}. Now all
  possibilities are excluded by axiom \ref{it:wss4}. The case where
  one of $\sigma$, $\tau$ is of type $\sD_2$ is treated in the first
  paragraph of the proof of \cite{KnopSRSV}*{Thm.~4.5}.
\end{proof}

Since all $\sigma\in\Sigma$ are sums of positive roots this implies
(see \cite{Bou}*{Chap. V, \S3.5, Lemme 3(ii)}):

\begin{corollary}\label{lemma:linindep}

  Let $(\Xi,\Sigma,S\p)$ be a \wss.. Then $\Sigma$ is linearly
  independent.

\end{corollary}

There are a couple of obvious ways to produce new \wsss. from old
ones.

\begin{itemize}

\item Given $\Sigma$ and $S\p$ there is a minimal choice for $\Xi$,
  namely $\Xi_{\min}:=\ZZ\Sigma$. A \wss. with $\Xi=\Xi_{\min}$ is
  called \emph{wonderful}. On the other hand, there is also a maximal
  choice, namely
  \begin{align}
    \Xi_{\max}:=\{\chi\in\Xiq\mid\ &\<\chi\mid\av\>=0
                                     \text{ for all }\a\in S\p\text{ and }\label{eq:Ximax}\\
                                   &\<\chi\mid\av-\bv\>=0\text{ for all }\a+\b\in\Sigma \text{ of
                                     type }\sD_2\}\nonumber
  \end{align}
  If $\Xi$ is any lattice with
  $\Xi_{\min}\subseteq\Xi\subseteq\Xi_{\max}$ then $(\Xi,\Sigma,S\p)$
  is a \wss.. Of particular interest is the saturation
  $\Xi_{\rm sat}:=(\Xi\otimes\QQ)\cap\Xiq$. It will describe the image
  of the dual group.

\item If $\Sigma_0\subseteq\Sigma$ is any subset then
  $(\Xi,\Sigma_0,S\p)$ is a \wss. which is called the
  \emph{localization in $\Sigma_0$}.

\item For any subset $S_0\subseteq S$ let
  $\cSq_0:=(\Xiq,S_0,\Xiqv,\Sv_0)$ (the based root datum corresponding
  to a Levi subgroup $L_0\subseteq G$). Then $(\Xi,\Sigma_0,S_0\p)$ is
  a \wss. for $\cSq_0$ where
  $\Sigma_0:=\{\sigma\in\Sigma\mid|\sigma|\subseteq S_0\}$ and
  $S_0\p:=S\p\cap S_0$. This \wss. is called the \emph{localization in
    $S_0$}.

\end{itemize}

Let $\Phi\p\subseteq\Phi$ be the root subsystem generated by $S\p$ and
let $\rho$, $\rho\p$ be the half-sum of positive roots of $\Phi$,
$\Phi\p$, respectively. Later we need:

\begin{lemma}\label{lemma:rho-rhoL}

  $\rho-\rho\p\in\half\Xi_{\max}$.

\end{lemma}

\begin{proof}

  It is well-known that $2\rho,2\rho\p\in\ZZ S$ and that
  \begin{equation}\label{eq:rho}
    \<\rho\mid\av\>=1\text{ for all $\a\in S$ and}
    \<\rho\p\mid\av\>=1\text{ for all $\a\in S\p$.}
  \end{equation}
  Hence $\<\rho-\rho\p|\av\>=0$ for all $\a\in S\p$. Let
  $\sigma=\g_1+\g_2\in\Sigma $ of type $\sD_2$. Then
  \begin{equation}
    \<\a\mid\gv_i\>=0\text{ for all $\a\in S\p$}
  \end{equation}
  implies
  $\<\rho-\rho\p|\gv_1-\gv_2\>=
  \<\rho|\gv_1-\gv_2\>-\<\rho\p|\gv_1-\gv_2\>=0.$
\end{proof}

Next we determine the relative position of any two spherical roots
$\sigma,\tau\in\Sigma$. More precisely, since
\begin{equation}\label{eq:rank2}
  (\ZZ\sigma+\ZZ\tau,\{\sigma,\tau\},(|\sigma|\cup|\tau|)\cap S\p)
\end{equation}
is a \wss. we are reduced to data with $S=|\sigma|\cup|\tau|$. Since
\ssys. of this type have been classified this poses the technical
question whether every \wss. actually comes from a proper one. The
following lemma shows that the answer is affirmative for wonderful
systems.

\begin{lemma}

  Let $\cS=(\Xi,\Sigma,S\p)$ be a \wss.. Then $\Xi$ contains a
  subgroup $\Xi_0$ of finite index such that $(\Xi_0,\Sigma,S\p)$ is
  induced by a \ssy.. If $\cS$ is wonderful one can take
  $\Xi_0=\Xi=\ZZ\Sigma$.

\end{lemma}
456
\begin{proof}

  We have to construct a
  \ssy. $\tilde\cS=(\tilde\Xi,\tilde\Sigma,\cD,c,M)$. Since we use its
  definition only in this proof we refrain from stating it
  here. Instead, we refer to the definition of a $p$-spherical system
  \cite{KnopLocalization}*{Def.~71} for $p=0$. There the axioms are
  labeled A1 through A8.

  First by \cref{lemma:linindep} there is a subgroup
  $\Xi_0\subseteq\Xi$ of finite index such that $\Sigma$ is part of a
  basis of $\Xi_0$.  Then put $\tilde\Xi:=\Xi_0$ and
  $\tilde\Sigma=\Sigma$. Since all elements of $\Sigma$ are primitive
  in $\Xi_0$, axiom A1 is satisfied. The axioms
  \ref{it:wss1}--\ref{it:wss3} now imply the corresponding axioms A3,
  A2, and A8 for $\tilde\cS$. Also, since $2S\cap\Sigma=\leer$, axiom
  A7 is vacuously satisfied. The other axioms A4, A5, and A6 deal with
  elements of $S^{(a)}:=\Sigma\cap S$. It has been shown in
  \cite{Luna} that it suffices to construct the set $\cD^{(a)}$ of all
  $D\in\cD^{(a)}$ for which there is an $\a\in S^{(a)}$ such that
  $(D,\a)\in M$. To this end we define $\cD^{(a)}$ formally as the
  disjoint union of pairs $\{D_\a^+,D_\a^-\}$ with $\a\in
  S^{(a)}$. For any $\a\in S^{(a)}$ we define $(D_\a^+,\b)\in M$ if
  and only if $\a=\b$. Finally we need to define the elements
  $c_\a^\pm:=c(D_\a^\pm)\in\Xiv$. To force A4, A5, and A5 to be true
  they have to satisfy
  \begin{equation}\label{eq:calpha}
    \begin{array}{l}
      c_\a^+(\chi)+c_\a^-(\chi)=\<\chi\mid\av\>,\\
      c_\a^\pm(\a)=1,\\
      c_\a^\pm(\sigma)\le0\text{ for }\sigma\in\Sigma\setminus\{\a\}.
    \end{array}
  \end{equation}
  Since $\Sigma$ is part of a basis of $\Xi_0$ there is
  $c_\a^+\in\Xiv_0$ with $c_\a^+(\sigma)=\d_{\a\sigma}$ (Kronecker
  $\d$). With $c_\a^-:=\av-c_\a^+$ the first two properties of
  \eqref{eq:calpha} hold while the third follows from
  \cref{lemma:le0}.

  The last assertion is clear since
  $\ZZ\Sigma\subseteq\Xi_0\subseteq\Xi$.
\end{proof}

\begin{remark}

  The change of $\Xi$ cannot be avoided. Take, e.g., $G=\SL(2)$ and
  $\cS=(\ZZ\omega,\{\a\},\leer)$.  The corresponding \ssy. would come
  from a non-horospherical, homogeneous $G$-variety, i.e., either
  $G/T$ or $G/N(T)$. But those have $\Xi=\ZZ(2\omega)$ and
  $\Xi=\ZZ(4\omega)$, respectively. Note, however, that $\cS$ does
  come from the non-spherical variety $X=\SL(2)/\{e\}$.
\end{remark}

From this we deduce:

\begin{theorem}

  \cref{tab:rank2} lists, up to isomorphism, all indecomposable
  \wsss. of type \eqref{eq:rank2}.

\end{theorem}

\begin{proof}

  We know that all wonderful \wsss. are induced from wonderful
  \ssys.. Now use Bravi's classification \cite{Bravi}*{Appendix A} of
  all such \ssys. of rank $2$. One can also use the earlier paper
  \cite{Wasserman} by Wasserman which classifies wonderful varieties
  of rank $2$. But then one has to rely on the non-trivial fact
  (proved in \cite{BraviPezzini}) that all \ssys. come from spherical
  varieties.
\end{proof}
\begin{remark}

  The underlinings and asterisks in \cref{tab:rank2} are used for the
  proof of \cref{thm:assexist}.

\end{remark}

\section{Associated roots}

Following \cite{SV}, we are going to associate certain roots (of $G$)
to every spherical root $\sigma\in\Sigma$. Observe that most spherical
roots (above the dividing line in \cref{tab:NSR}) are already
roots. In this case, $\sigma$ itself is the only root associated to
$\sigma$.

For non-roots we use the following result:

\begin{lemma}\label{lemma:assoc}

  Let $\sigma\in\Sigma\setminus\Phi$. Then there is a unique set
  $\{\g_1,\g_2\}$ of positive roots such that
  \begin{enumerate}

  \item\label{it:assoc1} $\sigma=\g_1+\g_2$,

  \item\label{it:assoc2} $\g_1$ and $\g_2$ are strongly orthogonal,
    i.e., $(\QQ\g_1+\QQ\g_2)\cap\Phi=\{\pm\g_1,\pm\g_2\}$,

  \item\label{it:assoc3} $\gv_1-\gv_2=\dv_1-\dv_2$ with
    $\d_1,\d_2\in S$.

  \end{enumerate}

\end{lemma}

\begin{proof}

  The existence of such a decomposition is established by the
  following table:
  \begin{equation}\label{tab:assocroots}
    \begin{array}{lllll}

      |\sigma|&\g_1,\g_2&\gv_1,\gv_2&\dv_1,\dv_2\\

      \noalign{\smallskip\hrule\smallskip}

      \sD_2
              &
                \a_1,\ \a_2
                        &
                          \av_1,\ \av_2
                                    &
                                      \av_1,\ \av_2\\
      \noalign{\smallskip\hrule\smallskip}

      \sD_n&(\a_1\+\a_{n-2})+\a_{n-1},
                        &
                          (\av_1\+\av_{n-2})+\av_{n-1},&\av_{n-1},\ \av_n\\
      n\ge3&(\a_1\+\a_{n-2})+\a_n&(\av_1\+\av_{n-2})+\av_n\\
      \noalign{\smallskip\hrule\smallskip}

      \sB_3&\a_1+\a_2+2\a_3,\ \a_2+\a_3
                        &
                          \av_1+\av_2+\av_3,\ 2\av_2+\av_3
                                    &
                                      \av_1,\ \av_2\\

      \noalign{\smallskip\hrule\smallskip}

    \end{array}
  \end{equation}
  Uniqueness follows by an easy case-by-case consideration.
\end{proof}

\begin{remarks}
  \textit{i)} It is easy to see that $\g_1,\g_2\in\Phi$ are strongly
  orthogonal if and only if there is $w\in W$ such that $w\g_1$,
  $w\g_2$ are orthogonal simple roots. Observe that then also the
  coroots $\gv_1$, $\gv_2$ are strongly orthogonal.

  \textit{ii)} The possibility to decompose spherical roots into two
  strongly orthogonal roots was first observed by Brion-Pauer
  \cite{BrionPauer}*{4.2~Prop.}. In this form, the construction is
  from \cite{SV}.

\end{remarks}

The roots $\g_1,\g_2$ will be called \emph{associated to $\sigma$}. We
also use the following notation:

\begin{definition}

  For $\sigma\in\Sigma$ let
  \begin{equation}
    \sa:=
    \begin{cases}
      \{\siv\}&\text{if }\sigma\in\Phi,\\
      \{\gv_1,\gv_2\}&\text{if $\sigma,\g_1,\g_2$ are as in
        \cref{lemma:assoc}}.
    \end{cases}
  \end{equation}
  More generally, for any subset $\Sigma_0\subseteq\Sigma$ let
  $\Sa_0:=\bigcup_{\sigma\in\Sigma_0}\sa$.

\end{definition}

We record some more properties of associated roots. Let $\g_1$ and
$\g_2$ be associated to the non-root $\sigma\in\Sigma$. Using the
$W$-invariant scalar product we define the coroot of $\sigma$ (as
usual) as
\begin{equation}
  \siv=\frac{2\sigma}{\|\sigma\|^2}
\end{equation}
Since $\g_1$ and $\g_2$ are orthogonal we have
$\|\sigma\|^2=\|\g_1\|^2+\|\g_2\|^2$ and therefore
$\siv=c_1\gv_1+c_2\gv_2$ with
\begin{equation}
  c_1+c_2=\frac{\|\g_1\|^2}{\|\sigma\|^2}+\frac{\|\g_2\|^2}{\|\sigma\|^2}=1.
\end{equation}
From this we deduce:

\begin{lemma}\label{lemma:sigmavee}

  Let $\sigma=\g_1+\g_2$ be as above and let $\chi\in\Xi$. Then
  \begin{equation}
    (\chi,\siv)=\<\chi\mid\gv_1\>=\<\chi\mid\gv_2\>.
  \end{equation}

\end{lemma}

\begin{proof}

  Let $\e:=\gv_1-\gv_2=\dv_1-\dv_2$. Then we claim that
  $\<\Xi\mid\e\>=0$. If $\sigma$ is of type $\sD_2$ then this is axiom
  \ref{it:wss3} (of \cref{def:WSS}). Otherwise both $\d_i$ are in
  $S\p$ (see \eqref{tab:assocroots} and \cref{tab:NSR}) and the claim
  follows from axiom \ref{it:wss2}. From the claim we get
  $\<\chi\mid\gv_1\>=\<\chi\mid\gv_2\>$ and therefore
  \begin{equation}
    (\chi,\siv)=
    c_1\<\chi\mid\gv_1\>+c_2\<\chi\mid\gv_2\>=(c_1+c_2)\<\chi\mid\gv_1\>=\<\chi\mid\gv_1\>.\qedhere\hfill\qed
  \end{equation}
\end{proof}

\begin{corollary}

  Let $\sigma\in\Sigma$. Then $\chi\to(\chi,\siv)$ is an element of
  $\Xiv$ (also denoted by $\siv$) which is independent of the choice
  of the scalar product.

\end{corollary}

\begin{proof}

  For $\sigma\in\Phi$ this is clear. Otherwise, the assertion follows
  from \cref{lemma:sigmavee}.
\end{proof}

Later on, we also need the following facts on spherical roots:

\begin{lemma}\label{lemma:addclosed2}
  
  Let $\g\in\sa$ and $\d\in S\p$. Then $\{\pm\gv,\pm\dv\}$ is an
  additively closed root subsystem of $\Phiv$ (and therefore
  $\<\g\mid\dv\>=0$) except for $\g=\g_i\in\sa$ with
  $\sigma\in\Sigma\setminus\Phi$ and $\d=\d_j$. Then
  $\<\g\mid\dv\>=(-1)^{i+j}$ and $\<\d_1+\d_2\mid\gv_i\>=0$.

\end{lemma}

\begin{proof}

  Let $\g\in\sa$ with $\sigma\in\Sigma$. Assume first that
  $\d\not\in|\sigma|$. Then axiom \ref{it:wss2} implies that $\d$ is
  orthogonal to every simple root in $|\sigma|$. It follows that $\d$
  is strongly orthogonal to every root whose support lies in
  $|\sigma|$. This holds for $\g$, in particular.

  If $\d\in|\sigma|$ then the assertion can be verified by going
  through \cref{tab:NSR} case-by-case:

  Assume that $\g=\sigma\in\Phi$ but $\sigma$ is not the second
  $\sG_2$-case. Since $\d$ is orthogonal to $\g$ in that case it
  suffices to show that $\siv+\dv\not\in\Phi$. But that follows from
  the fact that $\sigma$ is always the dominant short root. Hence
  $\siv$ is the highest root.

  If $\sigma$ is the second $\sG_2$-case or of type $\sD_2$ then the
  assertion is moot since then $S\p=\leer$.
  
  The remaining cases ($\sD_n$ and $\sB_3$) are now easily checked
  one-by-one.
\end{proof}

\begin{remark}

  In most cases, $\g$ and $\d$ are even strongly orthogonal but that
  is not always the case: if $\sigma$ is of the first $\sC_n$-type and
  $\d=\a_1$ then $\sigma+\d$ is a root while $\siv+\dv$ is not.

\end{remark}

\section{The dual and the associated group}

One of the most important features of a \wss. is:

\begin{theorem}\label{thm:rootsystem}

  Let $(\Xi,\Sigma,S\p)$ be a \wss.. Then $(\Xi,\Sigma,\Xiv,\Siv)$ is
  a based root datum.

\end{theorem}

\begin{proof}

  Follows from \cref{prop:cryst} and \cref{lemma:le0}.
\end{proof}

From this we get the main object of the paper:

\begin{definition}

  Let $\cS=(\Xi,\Sigma,S\p)$ be a \wss.. Then the connected reductive
  group over $\CC$ with based root datum $(\Xiv,\Siv,\Xi,\Sigma)$ is
  denoted by $\Gv_\cS$ and is called the \emph{dual group of $\cS$}.

\end{definition}

Our principal goal is to embed the dual group $\Gv_\cS$ (up to
isogeny) into the Langlands dual group $\Gv$. For this, we define an
intermediate group by noticing that also $\Sa$ is a basis of a root
system. More precisely:

\begin{theorem}\label{thm:assexist}

  Let $\cS=(\Xi,\Sigma,S\p)$ be a \wss.. Then there is a (unique)
  connected reductive subgroup $\Ga_\cS\subseteq \Gv$ containing
  $\Tv\subseteq \Gv$ such that $\Sa$ is its set of simple roots.

\end{theorem}

\begin{proof}

  A root subsystem of $\Phiv$ is the root system of a connected
  reductive subgroup of $\Gv$ containing $\Tv$ if and only if it is
  additively closed (see \cite{BdS}). So, by \cref{lemma:addclosed3},
  we have to show that $\e:=\gv_1-\gv_2\not\in(\Phiv)^+$ for all
  $\gv_1\ne\gv_2\in\Sa$.

  If both $\g_i$ are associated to the same $\sigma\in\Sigma$ then
  $\e=\dv_1-\dv_2$ where the $\dv_i$ are distinct simple roots of
  $\Gv$, hence $\e\not\in\Phiv$.

  Now assume that $\g_1,\g_2$ are associated to distinct elements
  $\sigma_1,\sigma_2\in\Sigma$, respectively. If $\e\in(\Phiv)^+$ then
  the supports have to be contained one in another:
  $|\gv_2|\subseteq|\gv_1|$. In \cref{tab:rank2} these cases are
  marked by an asterisk (for convenience, the non-roots are
  underlined). They are:

  $\begin{array}{llll}
     |\sigma_1|{\cup}|\sigma_2|&\sigma_1,\sigma_2&\gv_1,\gv_2&\e\\
     \noalign{\smallskip\hrule\smallskip}
     \sB_4&\7\a_1+\a_2+\a_3+\a_4,\a_2+2\a_3+3\a_4&\72\av_1+2\av_2+2\av_3+\av_4,\av_2+\av_3+\av_4&2\av_1+\av_2+\av_3\\

     \noalign{\smallskip\hrule\smallskip}
     \sB_4&\7\a_1+\a_2+\a_3+\a_4,\a_2+2\a_3+3\a_4&\72\av_1+2\av_2+2\av_3+\av_4,2\av_3+\av_4&2\av_1+2\av_2\\

     \noalign{\smallskip\hrule\smallskip}
     \sC_n&\7\a_1+2\a_2\+2\a_{n-1}+\a_n,\a_1&\7\av_1+2\av_2\+2\av_n,\av_1&2\av_2\+2\av_n\\
     \noalign{\smallskip\hrule\smallskip}
     \sC_n+\sA_1&\7\a_1+2\a_2\+2\a_{n-1}+\a_n,\a_1+\a_1'&\7\av_1+2\av_2\+2\av_n,\av_1&2\av_2\+2\av_n\\
     \noalign{\smallskip\hrule\smallskip}
     \sG_2&\7\a_1+\a_2,\a_1&\7\av_1+3\av_2,\av_1&3\av_2\\
     \noalign{\smallskip\hrule\smallskip}

\end{array}
$

In none of the cases is $\e$ a root of $\Gv$.
\end{proof}

\begin{remark}

  Observe that it is crucial to pass to the dual root system. Consider
  for example the first $\sC_n$-case in \cref{tab:rank2}. Then
  $\g_1-\g_2=\sigma_1-\sigma_2=2\a_2\+2\a_{n-1}+\a_n$ is actually a
  root of $\sC_n$. Therefore neither $\Sa$ nor $\Sigma$ is the set of
  simple roots for a subgroup of $G$.

\end{remark}

Let $A$ be the torus with character group $\Xi$. Then the dual torus
\begin{equation}
  \Av=\Hom(\Xiv,\G_m)=\Xi\otimes\G_m
\end{equation}
with character group $\Xiv$ is, by construction, a maximal torus of
the dual group $\Gv_\cS$. The inclusion $\Xi\to\Xiq$ induces a
homomorphism $\eta:\Av\to \Tv$.

\begin{definition}

  Let $\cS=(\Xi,\Sigma,S\p)$ be a \wss.. Then a homomorphism
  $\phi:\Gv_\cS\to \Gv$ is called \emph{adapted} if it factors through
  $\Ga_\cS\subseteq \Gv$ and if it is compatible with the map $\eta$
  between maximal tori, i.e., if the following diagram commutes:
  \begin{equation}
    \cxymatrix{\Av\ar[r]^\eta\ar@{^(->}[d]&\Tv\ar@{^(->}[d]\\
      \Gv_\cS\ar[r]^\phi&\Ga_\cS}
  \end{equation}

\end{definition}

To show the existence of adapted homomorphisms we observe that the
sets $\sa$, $\sigma\in\Sigma$, partition $\Sa$ into subsets of size at
most $2$. Therefore, there is a unique involution $s$ acting on $\Sa$
whose orbits are precisely the sets $\sa$. Our main observation is:

\begin{lemma}

  The action of $s$ on $\Sa$ is a folding in the sense of
  \cref{def:folding}.

\end{lemma}

\begin{proof}

  Folding property \ref{it:fold1} follows from
  \cref{lemma:assoc}\,\ref{it:assoc2} (strong orthogonality of
  $\gamma_1$ and $\gamma_2$). Property \ref{it:fold2} is trivial if
  $\av=\sigma\in\Sigma\cap\Phi$ since then $\a=\s\a$. Otherwise, it
  follows from \cref{lemma:sigmavee} since
  $\bv+\s\bv\in\Sigma\cup2\Sigma\subseteq\Xi$ for all $\b=\gv\in\Sa$.
\end{proof}

From this, we get:

\begin{theorem}\label{thm:main}

  Let $\cS=(\Xi,\Sigma,S\p)$ be a \wss.. Then there exists an adapted
  homomorphism $\phi:\Gv_\cS\to \Gv$.

\end{theorem}

\begin{proof}

  Apply \cref{lemma:folded} to the $s$-action on $\Sa$ and
  $r=\res:\Xiqv\to\Xiv$.
\end{proof}

\begin{remark}

  The kernel $K$ of $\eta:\Av\to\Tv$ is finite since
  $\Xiv(\Av)=\Xi\to\Xiq=\Xiv(\Tv)$ is injective. Hence $K$ is also the
  kernel of $\phi$. Moreover, it is easy to see that
  $K\cong\Xi_{\rm sat}/\Xi$ where
  $\Xi_{\rm sat}=(\Xi\otimes\QQ)\cap\Xiq$. This means that $\phi$ is
  injective if and only if $\Xi$ is a direct summand of $\Xiq$ or, put
  differently, that $\phi(\Gv_\cS)=\Gv_{\cS_{\rm sat}}$ where
  $\cS_{\rm sat}=(\Xi_{\rm sat},\Sigma,S\p)$.

\end{remark}

\cref{cor:minrank} now implies:

\begin{corollary}

  Let $\phi:\Gv_\cS\to\Gv$ be adapted. Then:

  \begin{enumerate}

  \item\label{it:main2} The variety $\Ga_\cS/\phi(\Gv_\cS)$ is affine
    spherical of minimal rank.

  \item\label{it:main3} $\phi$ induces an injective homomorphism
    $(\Gv_\cS)\ad\into(\Ga_\cS)\ad$ between adjoint groups. The
    variety $(\Ga_\cS)\ad/(\Gv_\cS)\ad$ is a product of varieties from
    the list \eqref{eq:minranklist}.

  \end{enumerate}

\end{corollary}

Next we address the uniqueness of adapted homomorphisms. This has been
already answered in \cite{SV} but we need a little extension. First
observe that, in general, adapted homomorphisms cannot be unique since
$\Ad(a)\circ\phi$ is again adapted whenever $\phi$ is adapted and
$a\in \Tv$. This $\Tv$-action is not free though. Therefore, we
consider the torus $\Ta\ad$ whose character group is $\ZZ\Sa$. Then
clearly $\Ta\ad$ is the maximal torus of the adjoint quotient
$(\Ga_\cS)\ad$ and therefore acts on $\Ga_\cS$ by automorphisms.

\begin{theorem}\label{thm:unique}

  \begin{enumerate}

  \item\label{it:unique1} Let $\cS_0=(\Xi,\Sigma_0,S\p)$ be the
    localization of $\cS$ with respect to a subset
    $\Sigma_0\subseteq\Sigma$ and let $\phi:\Gv_\cS\to\Gv$ be
    adapted. Then $\Gv_{\cS_0}\subseteq\Gv_\cS$ and
    $\res\phi:\Gv_{\cS_0}\to\Gv$ is adapted as well.

  \item\label{it:unique2} For $\sigma\in\Sigma$ let $\cS(\sigma)$ be
    the localization with $\Sigma_0=\{\sigma\}$. Then every system
    $(\phi_\sigma)_{\sigma\in\Sigma}$ of adapted homomorphisms
    $\Gv_{\cS(\sigma)}\to\Gv$ can be uniquely extended to an adapted
    homomorphism $\Gv_\cS\to\Gv$.

  \item\label{it:unique3} $\Ta\ad$ acts simply transitively on the set
    of adapted homomorphisms $\phi:\Gv_\cS\to\Gv$.

  \end{enumerate}

\end{theorem}

\begin{proof}

  Clearly, we have an inclusion $\Gv_{\cS_0}\subseteq\Gv_\cS$. It is
  easy to see that the folding process (\cref{lemma:folded}) commutes
  with restricting to an $s$-invariant subset $S_0$ of $S$. This shows
  that $\res\phi$ has values in $\Ga_{\cS_0}$ and is therefore
  adapted, proving \ref{it:unique1}.

  An adapted homomorphism $\phi_\sigma:\Gv_{\cS(\sigma)}\to\Gv$ is
  uniquely determined by the image of a generator
  $e_{\siv}\in\fgv_\cS$ in
  $\bigoplus_{\gv\in\sa}\fgv_{\gv}$. Moreover, the image vector has to
  have a non-zero component in every summand (since
  $\res_{\Av}\gv=\siv$ is non-trivial). Thus, the torus $\Ta\ad$ acts
  transitively on the set of all $\phi_\sigma$ where the action of
  $t\in\Ta\ad$ depends only on the character values $\gv(t)$ with
  $\gv\in\sa$. Since $\Sa$ is linearly independent, this implies that
  $\Ta\ad$ acts simply transitively on the set of families
  $(\phi_\sigma)_{\sigma\in\Sigma}$ of adapted homomorphisms. The
  existence of an adapted homomorphism shows that there is a family
  which has an extension $\phi$. So all have. Uniqueness follows from
  the fact that $\Gv_{\cS}$ is generated by the subgroups
  $\Gv_{\cS(\sigma)}$, $\sigma\in\Sigma$ proving \ref{it:unique2}.

  This, finally, implies that $\Ta\ad$ acts also simply transitively
  on the set of adapted maps $\phi$, proving \ref{it:unique3}.
\end{proof}

\section{Momentum maps and invariant differential
  operators}\label{sec:moment}

Some results from \cite{KnopAB} and \cite{KnopHC} can be reformulated
using the dual group or rather its Lie algebra. For this assume that
$k=\CC$. Then $\ft^\vee$ is is the same as the dual Cartan subalgebra
$\ft^*$ which is canonically a subspace of the coadjoint
representation.

Now let $X$ be a smooth $G$-variety. Then $X$ induces a \wss. $\cS$
(\cref{prop:non-spherical}). In the following, we replace the index
$\cS$ by $X$. So, the dual group of $X$ is $\Gv_X$. Its Weyl group is
$W_X$. Since we are only concerned with Lie algebras, the distinction
between the lattices $\Xi$ and $\tilde\Xi$ in \eqref{eq:non-spherical}
is irrelevant.

Let $m:T^*_X\to\fg^*$ be the moment map on the cotangent bundle.  It
was shown in \cite{KnopAB} that there is a canonical surjective
$G$-invariant morphism $m_0:T^*_X\to\fa^*/W_X$ with irreducible
generic fibers such that the right hand square of following diagram
commutes:
\begin{equation}\label{eq:CDTvee}
  \cxymatrix{
    \fgv_X\ar@{>>}[r]\ar@{^(->}[d]&\fa^\vee/W_X\ar@{=}[r]\ar[d]&
    \fa^*/W_X\ar[d]&T^*_X\ar[d]^m\ar@{>>}_<<<{m_0}[l]\\
    \fgv\ar@{>>}[r]&\ft^\vee/W\ar@{=}[r]& \ft^*/W&\fg^*\ar@{>>}[l] }
\end{equation}
The left hand square combines two applications of the Chevalley
isomorphism. Therefore the whole diagram commutes.

If $X$ is affine and spherical then $m_0$ is the categorical quotient
and the diagram can be interpreted as follows:

\begin{itemize}

\item There is a bijective correspondence between semisimple conjugacy
  classes of $\fgv$ and closed coadjoint orbits in $\fg^*$.

\item There is a bijective correspondence between semisimple conjugacy
  classes of $\fgv_X$ and closed $G$-orbits in $T^*_X$.

\item The moment map is compatible with these correspondences. In
  particular, if $o^\vee\subseteq\fgv$ corresponds to
  $o^*\subseteq\fg^*$ then the symplectic reduction
  $m^{-1}(o^*)\mod G$ (a finite set) is in bijection with
  $(o^\vee\cap\fgv_X)/\Gv_X$.

\end{itemize}

There is also a non-homogeneous version of this theorem which is
closer to representation theory. It works by replacing the cotangent
bundle by the ring $\cD(X)$ of differential operators on $X$. In
\cite{KnopHC} a certain central subalgebra $\cZ(X)$ of $\cD(X)$ was
constructed which, in case $X$ is affine, even coincides with the
center. On the other hand if $X$ is spherical (but possibly
non-affine) then $\cD(X)=\cZ(X)$ is commutative.  Let
$Z_X=\Spec\cZ(X)$ and $Z_G:=\Spec\cZ(G)$ where $\cZ(G)$ is the center
of the enveloping algebra $\cU(\fg)$. The main result of \cite{KnopHC}
is a Harish-Chandra type isomorphism
$Z(X)\overset\sim\to(\fa^*+\rho)/W_X$ where $\rho$ is the half-sum of
positive roots such that the right hand square of the following
diagram commutes:
\begin{equation}\label{eq:CDZvee}
  \cxymatrix{
    \fgv_X+\rho\ar@{>>}[r]\ar@{^(->}[d]&(\fa^\vee+\rho)/W_X\ar@{=}[r]\ar[d]&
    (\fa^*+\rho)/W_X\ar[d]&Z_X\ar[d]\ar[l]_>>>\sim\\
    \fgv\ar@{>>}[r]&\ft^\vee/W\ar@{=}[r]& \ft^*/W&Z_G\ar[l]_\sim}
\end{equation}
For the top left map to make sense and the diagram to commute, we need
the following:

\begin{lemma}

  The subspace $\fgv_X+\CC\rho\subseteq\fgv$ is a reductive
  subalgebra. In particular, the affine hyperplane $\fgv_X+\rho$ is
  $\Ad \Gv_X$-invariant.

\end{lemma}

\begin{proof}

  With the notation of \cref{lemma:rho-rhoL} we have
  $\rho_0:=\rho-\rho\p\in\half\Xi_{\max}$. Thus, $\fgv_X+\CC\rho_0$ is
  the dual subalgebra corresponding to the character group
  $\Xi+\ZZ(2\rho_0)$. Since $S$ is orthogonal to $\Xi_{\max}$ also
  $(\fgv_X+\CC\rho_0)\oplus\CC\rho\p$ is a subalgebra which contains
  $\fgv_X$ as an ideal with abelian quotient. This implies the
  assertion.
\end{proof}

If $X$ is spherical or affine then we have:

\begin{itemize}

\item There is a bijective correspondence between semisimple conjugacy
  classes of $\fgv$ and central characters of $\cU(\fg)$.

\item There is a bijective correspondence between semisimple conjugacy
  classes of $\fgv_X+\rho$ and central characters of $\cD(X)$.

\item These correspondences are compatible. In particular, if
  $o^\vee\subseteq\fgv$ corresponds to the central character $\chi$ of
  $\cU(\fg)$ then there is a bijective correspondence between the set
  of extensions of $\chi$ to a central character of $\cD(X)$ and the
  set $(o^\vee\cap(\fgv_X+\rho))/\Gv_X$.

\end{itemize}

\section{Centralizers}\label{sec:centralizer}

Let $\Phi\p=\Phi\cap\ZZ S\p\subseteq\Phi$ be the root subsystem
generated by $S\p$. Then all roots in $\Phi\p$ are orthogonal to
$\Xi$. Sometimes, a converse is true:

\begin{definition}\label{def:nondeg}

  A \wss. $(\Xi,\Sigma,S\p)$ is \emph{non-degenerate} if $\a\in\Phi$
  and $\<\Xi\mid\av\>=0$ imply $\a\in\Phi\p$.

\end{definition}

A similar condition has been considered in \cite{KnopAB} along with
the following remark:

\begin{lemma}\label{lemma:nondeg}

  Let $(\Xi,\Sigma,S\p)$ be a \wss.. Then $(\Xi_{\max},\Sigma,S\p)$ is
  non-degenerate.

\end{lemma}

\begin{proof}

  Since $\<\rho-\rho\p\mid\av\>=0$ implies $\a\in\Phi\p$ (see
  \eqref{eq:rho}) the assertion follows from \cref{lemma:rho-rhoL}.
\end{proof}

\begin{remark}

  The simplest example of a degenerate \wss. is $\cS=(0,\leer,\leer)$
  when $S\ne\leer$. It corresponds to the full flag variety $G/B$.

\end{remark}

Let $\cS=(\Xi,\Sigma,S\p)$ be a \wss. and let $W_\cS$ be the Weyl
group of the corresponding root system
(\cref{thm:rootsystem}). Clearly, it is also the Weyl group of the
dual group $\Gv_\cS$. A priori, $W_\cS$ acts only on $\Xi$. Next we
show that this action extends in a natural way to all of $\Xiq$. To
this end we define for $\sigma\in\Sigma$:
\begin{equation}\label{eq:nsigma}
  n_\sigma:=
  \begin{cases}
    s_\sigma&\text{if $\sigma\in\Sigma\cap\Phi$}\\
    s_{\g_1}s_{\g_2}&\text{if }\sigma=\g_1+\g_2\in\Sigma\setminus\Phi
  \end{cases}
\end{equation}

\begin{proposition}

  The map $s_\sigma\mapsto n_\sigma$ extends (uniquely) to a
  homomorphism $n:W_\cS\to W$. It has the following properties:
  \begin{itemize}

  \item $n(W_\cS)$ extends the $W_\cS$-action on $\Xi$, i.e.
    $n(w)\chi=w\chi$ for all $\chi\in\Xi$ and $w\in W_\cS$.

  \item $n(W_\cS)$ acts on $S\p\subseteq\Xiq$. More precisely, if
    $\sigma\in\Sigma$ then $n_\sigma=n(s_\sigma)$ acts on $\d\in S\p$
    as
    \begin{equation}
      n_\sigma(\d)=\begin{cases} \d_{3-i}&\text{if
        }\sigma=\g_1+\g_2\in\Sigma\setminus\Phi
        \text{ and }\d=\d_i\\
        \d&\text{otherwise.}
      \end{cases}
    \end{equation}

  \end{itemize}

\end{proposition}

\begin{proof}

  Let $\sigma\in\Sigma$. If $\sigma\in\Phi$ then
  $n_\sigma=s_\sigma$. Hence also $n_\sigma=s_\sigma$ acts on $\Xi$
  and $n_\sigma(\d)=\d$ for $\d\in S\p$ since then $\d\perp\sigma$.

  If $\sigma\not\in\Phi$ let $\sigma=\g_1+\g_2$ be the decomposition
  of \cref{lemma:assoc}. Let $\chi\in\Xi$. Then \cref{lemma:sigmavee}
  implies
  \begin{equation}
    n_\sigma(\chi)=s_{\g_1}s_{\g_2}(\chi)=
    \chi-\<\chi\mid\gv_1\>\g_1-\<\chi\mid\gv_2\>\g_2=
    \chi-\<\chi\mid\siv\>(\g_1+\g_2)=s_\sigma(\chi).
  \end{equation}
  Now let $\d\in S\p$. If $\d\not\in\{\d_1,\d_2\}$ then
  $\d\perp\g_1,\g_2$ (see \cref{lemma:addclosed2}) and therefore
  $n_\sigma(\d)=\d$. On the other hand, if $\d=\d_1$ then
  $\<\d\mid\gv_1\>=1$, $\<\d\mid\gv_2\>=-1$ and therefore
  \begin{equation}
    n_\sigma(\d)=\d_1-\g_1+\g_2=\d_2.
  \end{equation}

  The assertion that $s_\sigma\mapsto n_\sigma$ extends to a
  homomorphism does not depend on the choice of $\Xi$. So, we choose
  the maximal one $\Xi=\Xi_{\max}$.

  Now let $N\subseteq W$ be the subgroup of all $w\in W$ with
  $w\Xi=\Xi$, $\res_\Xi w\in W_\cS$ and $wS\p=S\p$. Then $\res_\Xi$
  induces homomorphism $N\to W_\cS$. This homomorphism is surjective,
  since $n_\sigma\in N$ with $\res_\Xi n_\sigma=s_\sigma$ and since
  $W_\cS$ is generated by all $s_\sigma$.

  We show that this homomorphism is injective. For this, let $w\in N$
  be in the kernel. Then the non-degeneracy of $\cS$
  (\cref{lemma:nondeg}) implies that $w\in W_{S\p}$, the group
  generated by all $s_\d$, $\d\in S\p$. Finally $wS\p=S\p$ forces
  $w=e$.

  Thus, we have shown that $\res_\Xi:N\to W_\cS$ is an isomorphism.
  Its inverse $n$ has all required properties.
\end{proof}

Our goal is to determine the centralizer of
$\phi(\Gv_\cS)\subseteq \Gv$. Observe that $S\p$ determines a Levi
subgroup $\LvS\subseteq \Gv$ and the action of $W_\cS$ on $S\p$
induces an action on $\LvS$ by permuting the generators
$e_{\dv}\in\fgv$, $\d\in S\p$. It will turn out that the fixed point
group $(\LvS)^{W_\cS}$ is almost the centralizer of
$\phi(\Gv_\cS)$. But first, we look at its structure:

\begin{proposition}\label{prop:LWList}

  Up to isogeny, the inclusion $(\LvS)^{W_\cS}\subseteq \LvS$ is
  isomorphic to a product of
  \begin{equation}\label{eq:WS-List}
    \begin{array}{ll}
      \bullet\quad 1\subseteq \G_m,\\
      \bullet\quad H\subseteq H&\text{with $H$ simple},\\
      \bullet\quad \SL(2)\subseteq\SL(2)^n&n\ge2,\\
      \bullet\quad \SO(2n-1)\subseteq\SO(2n)&n\ge3,\\
      \bullet\quad \sG_2\subseteq\SO(8),\\
      \bullet\quad \SO(3)\subseteq\SL(3).
    \end{array}
  \end{equation}

\end{proposition}

\begin{proof}

  Let $\Sigma'\subseteq\Sigma$ be the set of all $\sigma\in\Sigma$ of
  type $\sD_{n\ge3}$ or $\sB_3$ and let $S'\subseteq S$ be the union
  of their supports. The complement $S\p\setminus S'$ gives rise to
  factors of the form $H\subseteq H$ because $W_\cS$ acts trivially on
  it. Thus, we may assume that $\Sigma=\Sigma'$ and $S=S'$.

  To check how two of the roots can be combined, we look at
  \cref{tab:rank2}. There, we find only two indecomposable rank-2-data
  where both roots are of type $\sD_{n\ge3}$ or $\sB_3$. These are
  supported on $\sA_5$ (with two roots of type $\sD_3$) and $\sE_6$
  (with two roots of type $\sD_5$), respectively. This implies easily
  that every connected component of $\Sv$ is one of the following:
  \begin{equation}
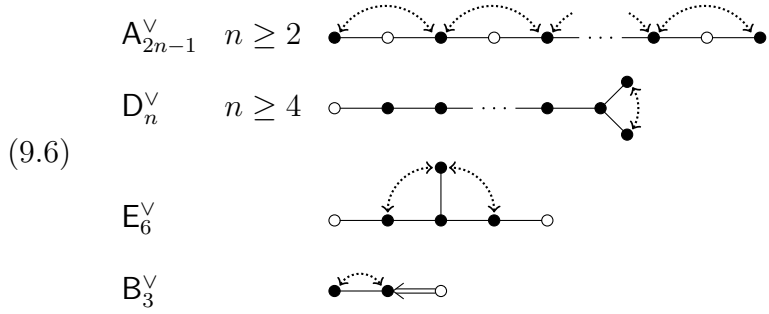
    \begin{array}{lll}
      \sA_{2n-1}^\vee&n\ge2&
                             \begin{tikzpicture}[scale=\scalar]
                               \xp(0,0) \draw (0.1,0)--(0.9,0);
                               \op(1,0) \draw (1.1,0)--(1.9,0);
                               \xp(2,0) \draw (2.1,0)--(2.9,0);
                               \op(3,0) \draw (3.1,0)--(3.9,0);
                               \xp(4,0) \ddd(4,0) \xp(6,0) \draw
                               (6.1,0)--(6.9,0); \op(7,0) \draw
                               (7.1,0)--(7.9,0); \xp(8,0) \draw
                               [<->,densely dotted, thick] (0.1,0.1)
                               arc (150:30:1); \draw [<->,densely
                               dotted, thick] (2.1,0.1) arc
                               (150:30:1); \draw [<-,densely dotted,
                               thick] (4.1,0.1) arc (150:120:1); \draw
                               [<-,densely dotted, thick] (5.9,0.1)
                               arc (30:60:1); \draw [<->,densely
                               dotted, thick] (6.1,0.1) arc
                               (150:30:1);
                             \end{tikzpicture}
      \\
      \sD_n^\vee&n\ge4&
                        \vcenter{\hbox{$\begin{tikzpicture}[scale=\scalar]
                              \draw (0,1) circle (0);
                              \draw (0,-1) circle (0);
                              \op(0,0)
                              \draw (0.1,0)--(0.9,0);
                              \xp(1,0)
                              \draw (1.1,0)--(1.9,0);
                              \xp(2,0)
                              \ddd(2,0)
                              \xp(4,0)
                              \draw (4.1,0)--(4.9,0);
                              \xp(5,0)
                              \draw (5,0)--(5.5,0.5);
                              \draw (5,0)--(5.5,-0.5);
                              \xp(5.5,0.5)
                              \xp(5.5,-0.5)
                              \draw [<->,densely dotted, thick] (5.6,0.4) arc (30:-30:0.8);
                            \end{tikzpicture}$}}
      \\
      \sE_6^\vee&&
                   \begin{tikzpicture}[scale=\scalar]
                     \op(0,0) \draw (0.1,0)--(0.9,0); \xp(1,0) \draw
                     (1.1,0)--(1.9,0); \xp(2,0) \draw
                     (2.1,0)--(2.9,0); \xp(3,0) \draw
                     (3.1,0)--(3.9,0); \op(4,0) \xp(2,1) \draw
                     (2,0)--(2,1); \draw [<->,densely dotted, thick]
                     (1,0.15) arc (180:90:0.85); \draw [<->,densely
                     dotted, thick] (3,0.15) arc (0:90:0.85);
                   \end{tikzpicture}
      \\
      \sB_3^\vee&&
                   \begin{tikzpicture}[scale=\scalar]
                     \draw (0,1) circle (0); \xp(0,0) \draw
                     (0,0)--(1,0); \xp(1,0) \draw
                     (1.15,0.04)--(1.9,0.04); \draw
                     (1.15,-0.04)--(1.9,-0.04); \op(2,0) \draw (1.1,0)
                     -- (1.3,0.15); \draw (1.1,0) -- (1.3,-0.15);
                     \draw [<->,densely dotted, thick] (0.1,0.1) arc
                     (150:30:0.45);
                   \end{tikzpicture}
      \\
    \end{array}
  \end{equation}
  Here the action by the simple generators $n_\sigma\in W_\cS$ on
  $S\p$ is indicated by dotted arrows. Now each of these diagrams
  gives rise to one case of \eqref{eq:WS-List}.
\end{proof}

The group $(\LvS)^{W_\cS}$ is slightly too big since, in general, not
even $(\Tv)^{W_\cS}$ centralizes $\Gv_\cS$. To be precise, the
character group of $(\Tv)^{W_\cS}$ is
$\Xi((\Tv)^{W_\cS})=\Xiqv/\Lambda_0$ where $\Lambda_0$ is the group
generated by all $\chiv-w\chiv$ with $\chiv\in\Xiqv$ and $w\in
W_\cS$. Then the equality
\begin{equation}
  \chiv-n_\sigma(\chiv)=
  \begin{cases}
    \<\sigma\mid\chiv\>\siv&\text{if }\sigma\in\Sigma\cap\Phi\\
    \<\g_1\mid\chiv\>\gv_1+\<\g_2\mid\chiv\>\gv_2& \text{if
    }\sigma=\g_1+\g_2\in\Sigma\setminus\Phi.
  \end{cases}
\end{equation}
shows that $\Lambda_0$ is of finite index in $\ZZ\Sa$. Thus,
$T^{\Sa}\subseteq \Tv$, the subgroup with character group
$\Xiqv/\ZZ\Sa$, is of finite index in $(\Tv)^{W_\cS}$. Clearly it
equals the center of $\Ga_\cS$ and it is also the centralizer of
$\phi(\Gv_\cS)$ in $\Tv$.

\begin{lemma}\label{lemma:Lstructure}

  There is a unique subgroup $\LaS\subseteq (\LvS)^{W_\cS}$ of finite
  index such that
  \begin{equation}
    \Tv\cap \LaS=T^{\Siv}.
  \end{equation}
  Moreover, $\LaS$ is reductive with maximal torus
  $(T^{\Siv})^\circ=((\Tv)^{W_\cS})^\circ$. Its derived subgroup $L_0$
  is semisimple with $\LaS=T^{\Siv}\,L_0$. The simple roots of $L_0$
  are the restrictions of the simple roots $\dv$, $\d\in S\p$ of
  $\LvS$. Two restrictions are equal if and only if they lie in the
  same $W_\cS$-orbit.

\end{lemma}

\begin{proof}

  Let $\Lq:=((\LvS)^{W_\cS})^\circ$. Then \cref{prop:LWList} entails
  that the normalizer of $\Lq$ is generated by $\Lq$ and the center of
  $\LvS$. This implies that $(\LvS)^{W_\cS}$ itself is generated by
  $\Lq$ and $\Tv\cap (\LvS)^{W_\cS}=(\Tv)^{W_\cS}$.

  Now let $\gv\in\Xiqv=\Xi(\Tv)$ be a character and let $\g_0$ be its
  restriction to $(\Tv)^{W_\cS}$. Then $\g_0$ extends to a character
  of $\Lq$ if and only if $\gv$ is trivial on the maximal torus of the
  semisimple part $\Lq'$ of $\Lq$. That torus is generated by the
  images of the simple coroots which are orbit sums
  $\overline\d:=\sum W_\cS\d$ with $\d\in S\p$. Therefore $\g_0$
  extends if and only if $\<\overline\d\mid\gv\>=0$ for all
  $\d\in S\p$. We claim that this condition holds for
  $\gv\in\Sa$. Indeed, this is clear if $\g\in S$. Otherwise,
  $\gv=\gv_i\in\sa$ for some $\sigma\in\Sigma$. Now the assertion
  follows from \cref{lemma:addclosed2}.

  Thus we have shown that every element of $\ZZ\Sa$ extends to a
  character of $(\LvS)^{W_\cS}$ and we can define
  $\LaS\subseteq (\LvS)^{W_\cS}$ to be the common kernel of these
  characters.

  The rest of the assertions are now clear or well-known.
\end{proof}

Now we have:

\begin{theorem}\label{thm:centralizer}
  
  Let $\cS=(\Xi,\Sigma,S\p)$ be a \wss.. Then there is an adapted
  homomorphism $\phi:\Gv_\cS\to \Gv$ such that $\phi(\Gv_\cS)$ and
  $\LaS$ centralize each other.

\end{theorem}

\begin{proof}

  For $\sigma\in\Sigma$ consider the rank-$1$-subgroup
  $F_\sigma:=\Gv_\cS(\siv)$. We first show that there is an adapted
  homomorphism $\phi_\sigma:F_\sigma\to \Gv$ whose image commutes with
  $\LaS$. For this let $L^\sigma:=\La_{\cS_\sigma}\subseteq\LvS$ be
  the subgroup corresponding to the localized system
  $\cS_\sigma:=(\Xi,\{\sigma\},S\p)$. Since $\LaS\subseteq L^\sigma$
  it suffices to find $\phi_\sigma$ whose image commutes with
  $L^\sigma$. We already know that $\phi_\sigma(F_\sigma)$ commutes
  with $\Tv\cap L^\sigma$. For $\d\in S\p$ let
  $H_\d\subseteq L^\sigma$ be the semisimple rank-1-subgroup whose
  positive root is the restriction $\bar\dv$ of $\dv$ to
  $\Tv\cap L^\sigma$ (see \cref{lemma:Lstructure}). Thus we have to
  show that $\phi_\sigma(F_\sigma)$ commutes with $H_\d$.

  Assume first that either $\sigma\in\Phi$ or that $\d\ne\d_i$ if
  $\sigma\not\in\Phi$. Then $H_\d$ commutes with all rank-1-subgroups
  $\Gv(\gv)$ with $\gv\in\sa$ by \cref{lemma:addclosed2}. Since
  $\phi_\sigma(F_\sigma)\subseteq\prod_{\gv\in\sa} \Gv(\gv)$, this
  implies that $\phi_\sigma(F_\sigma)$ commutes with $H_\d$.

  So assume that $\sigma=\g_1+\g_2\not\in\Phi$ and $\d=\d_i\in
  S\p$. Then there are the following two cases:

  The spherical root $\sigma$ is of type $\sD_n$ with $n\ge3$: then
  one verifies that the root subsystem of $\Phiv$ which is generated
  by $\gv_1$, $\gv_2$, $\dv_1$, and $\dv_2$ is of type $\sA_3$ with
  simple roots $\dv_1$, $\bv:=\gv_1-\dv_1=\gv_2-\dv_2$, and
  $\dv_2$. This root system is additively closed in $\Phiv$. Thus, it
  corresponds to a subgroup $J$ of $\Gv$ which is isogenous to
  $\SL(4)$. Let $U=\Span_\CC(u_1,u_2)$ and $V=\Span_\CC(v_1,v_2)$ be
  two copies of the defining representation of $\SL(2)$. Then the
  $\SL(2)\times\SL(2)$-action on $U\otimes V$ defines a homomorphism
  $\SL(2)\times\SL(2)\to\SL(4)$. More precisely, this homomorphism
  depends on the choice of an ordered basis which we take
  ($u_1\otimes v_1$, $u_1\otimes v_2$, $u_2\otimes v_1$,
  $u_2\otimes v_2$). Now one checks that the first factor is mapped to
  the product $J_{\gv_1}J_{\gv_2}=J_{\e_1-\e_3}J_{\e_2-\e_4}$ (where
  the $\e_i$ denote the canonical weights of the $\SL(4)$-module
  $\CC^4$). So this homomorphism is adapted. The second factor is
  mapped to the diagonal of
  $J_{\dv_1}J_{\dv_2}=J_{\e_1-\e_2}J_{\e_3-\e_4}$ which therefore
  equals $H_\d$. Both factors commute, which proves the assertion.

  The second case is when $\sigma$ is of type $\sB_3$. The additively
  closed subsystem of $\Phiv$ which is generated by $\gv_1$, $\gv_2$,
  $\dv_1$, and $\dv_2$ is the dual root system $\sC_3$. Thus $\Gv$
  contains a subgroup $J$ which is isogenous to $\Sp(6)$. Now consider
  the $2$-dimensional $\SL(2)$-representation $U=\Span_\CC(u_1,u_2)$
  and the $3$-dimensional $\SO(3)$-representation
  $V=\Span_\CC(v_1,v_2,v_3)$ (leaving invariant the quadratic form
  $q(x_1,x_2,x_3)=2x_1x_3-x_2^2$). The choice of the basis
  \begin{equation}
    u_1\otimes v_1, u_1\otimes v_2, u_1\otimes v_3, u_2\otimes v_1,
    u_2\otimes v_2, u_2\otimes v_3\in U\otimes V
  \end{equation}
  defines a homomorphism $\SL(2)\times\SO(3)\to\Sp(6)\subseteq\SL(6)$
  where the symplectic group is defined with respect to the
  skew-symmetric matrix ${\rm antidiag}(1,-1,1,-1,1,-1)$. Now one
  checks again that the first factor is mapped to the product
  $J_{\gv_1}J_{\gv_2}=J_{\e_1+\e_3}J_{2\e_2}$ (with $\e_i$ being
  weights of $\Sp(6)$) while the second factor goes to
  $H_\d\subseteq\GL(3)\subseteq\Sp(6)$.

  This finishes the proof of the assertion that for every $\sigma$
  there is $\phi_\sigma$ such that $\phi_\sigma(F_\sigma)$ commutes
  with $\LaS$. But, as we showed in
  \cref{thm:unique}\,\ref{it:unique2}, any family of adapted
  homomorphisms $(\phi_\sigma)_{\sigma\in\Sigma}$ can be extended to a
  unique adapted homomorphism $\phi$. Then $\LaS$ commutes with all
  subgroups $\phi(F_\sigma)$ and therefore with $\phi(\Gv_\cS)$.
\end{proof}

\begin{definition}

  A homomorphism satisfying the assertion of \cref{thm:centralizer}
  will be called \emph{very adapted}.

\end{definition}

It is easy to see from the proof that all other very adapted
homomorphisms are of the form $\Ad(t)\circ\phi$ where $t\in \Tv\ad$
with $\gamma_1(t)=\gamma_2(t)$ for all roots
$\sigma=\gamma_1+\gamma_2$ of type $\sD_{n\ge3}$ or $\sB_3$.

\begin{corollary}\label{cor:GZC}
  Let $\phi:\Gv_\cS\to\Gv$ be a very adapted homomorphism. Then the
  map
  \begin{equation}
    \Gv_\cS\times^{Z(\Gv_\cS)}\LaS \to \Gv:[g,l]\mapsto\phi(g)l
  \end{equation}
  is an injective homomorphism where $Z(\Gv_\cS)$ is the center of
  $\Gv_\cS$ (with character group $\Xiv/\ZZ\Siv$).

\end{corollary}

\begin{proof}

  Follows from \cref{thm:centralizer} and
  \cref{cor:minrank}\ref{it:minrank1}.
\end{proof}

In \cite{SV}, Sakellaridis-Venkatesh conjecture that the image of the
dual group is centralized by the so-called \emph{principal $\SL(2)$}
of $\Lv_\cS$. To establish this, that $\Phi\p$ is the root system of
$L_{\cS}$ and that $\rho\p\in\Xiq$ is the half-sum of positive
roots. Since $2\rho\p\in\Xiq$ we may regard it as a $1$-parameter
subgroup $\G_m\to \Tv$. It is well-known that there is a homomorphism
$\psi:\SL(2)\to \LvS$, called the principal $\SL(2)$, such that
$\res_{\G_m}\psi=2\rho\p$. This homomorphism is unique up to
composition with $\Ad(t)$, $t\in \Tv$. It will be normalized by
requiring that $\psi$ maps
$e=\begin{pmatrix}0&1\\0&0\end{pmatrix}\in\sl(2)$ to
$e_L:=\sum_{\delta\in S\p}e_{\dv}$.

\begin{proposition}\label{cor:arthur}

  Let $\psi:\SL(2)\to \LvS$ be the principal $SL(2)$ and let
  $\phi:\Gv_\cS\to\Gv$ be adapted. Then
  \begin{equation}
    \Gv_\cS\times\SL(2)\to\Gv:(g,l)\mapsto\phi(g)\psi(l)
  \end{equation}
  is a group homomorphism if and only if $\phi$ is very adapted.

\end{proposition}

\begin{proof}

  Clearly both $2\rho\p$ and $e_L$ are $W_\cS$-invariant which implies
  that $\psi$ factors through $\LaS$. Hence, if $\phi$ is very
  adapted, the assertion holds by definition.

  Conversely, assume that $\phi(\Gv_\cS)$ and $\psi(\SL(2))$ commute
  with each other. We have to show that then $\phi(\Gv_\cS)$ commutes
  with $\LaS$. Since $\phi(\Gv_\cS)$ centralizes
  $T^{\Siv}=\Tv\cap\LaS$ it suffices to show that it centralizes the
  semisimple part $L_0:=(\LaS)'$, as well. From the construction of
  $\LaS$ it follows that the simple root vectors of $L_0$ are
  $W_\cS$-orbit sums. Their sum is therefore the sum of all simple
  root vectors of $L_\cS$. This implies that $\psi$ is also a
  principal $\SL(2)$ with respect to $L_0$.

  Now let $\beta$ be a simple root of $L_0$. Then there is $n\ge1$ and
  a coweight $\omega^\vee:\G_m\to\Tv\cap L_0$ which is $n$ times a
  fundamental coweight, i.e., with
  $\<\beta'\mid\omega_\beta\>=n\delta_{\beta\,\beta'}$ for all simple
  roots $\beta'$ of $L_0$. Let $e:=\sum_\beta e_\beta$. Then
  $\phi(\Gv_\cS)$ will centralize $t^{-n}\omega^\vee(t)e$ for all
  $t\in\G_m$ and therefore it limit for $t\to0$ which is
  $e_\beta$. The same argument works for the negative simple root
  vectors $e_{\beta}$. Hence $\phi(\Gv_\cS)$ centralizes the Lie
  algebra of $L_0$ and therefore $L_0$, itself.
\end{proof}

\begin{remark}
  
  Let $\psi:\SL(2)\to \LvS$ be any homomorphism with
  $\psi(\G_m)\subseteq \Tv$ and let $\rho\in\Xiq$ be the corresponding
  $1$-parameter subgroup. Then $\psi$ is uniquely determined by the
  numbers $(\<\rho\mid\dv\>)_{\delta\in S\p}$, the so-called Dynkin
  characteristic of $\psi$. The characteristic of the principal
  $\SL(2)$ is $2$ for all $\delta$. Now clearly the same argument
  works for any $\psi$ whose characteristic is $W_\cS$-invariant.

\end{remark}

Next we address the problem of when $\LaS$ is the full centralizer of
$\Gv_\cS$ in $\Gv$. For this we need the non-degeneracy condition
(cf.\ \cref{def:nondeg}).

\begin{theorem}\label{thm:fullcentralizer}

  Let $\cS=(\Xi,\Sigma,S\p)$ be a non-degenerate \wss. and let
  $\phi:\Gv_\cS\to\Gv$ be very adapted. Then:

  \begin{enumerate}

  \item\label{it:full1} The centralizer of $\phi(\Av)$ in $\Gv$ is
    $\LvS$.

  \item\label{it:full2} The centralizer of $\phi(\Gv_\cS)$ in $\Gv$ is
    $\LaS$.

  \end{enumerate}

\end{theorem}

\begin{proof}

  Part \ref{it:full1} is more or less the definition of
  non-degeneracy. Let now $C\subseteq\Gv$ be the centralizer of
  $\phi(\Gv_\cS)$. Then $C\subseteq \LvS$ by \ref{it:full1}. Let
  $\sigma\in\Sigma$ and let $\tilde s_\sigma$ be any lift of
  $s_\sigma\in W_\cS$ to $\Gv_\cS$. The image
  $\tilde n_\sigma:=\phi(\tilde s_\sigma)$ lies diagonally in the
  product of all subgroups $\Gv(\gv)$ with $\gv\in\sa$. This implies
  that $\tilde n_\sigma$ is in fact also a lift of $n_\sigma$ in
  $\Gv$. Therefore, $\Ad\tilde n_\sigma$ permutes $S\p$ and hence
  normalizes $\LvS$. On the other hand, $\Ad\tilde n_\sigma$
  centralizes $\LaS$. In particular, it centralizes all orbit sums
  $\sum_{\overline\delta\in W_\cS\delta}e_{\overline\d^\vee}$ with
  $\delta\in S\p$. This shows that $\Ad\tilde n_\sigma$ acts in fact
  as a graph automorphism on $\LvS$. Applying this to all
  $\sigma\in\Sigma$ we see that $C\subseteq (\LvS)^{W_\cS}$. Finally,
  $C=\LaS$ follows from the fact that the centralizer of
  $\phi(\Gv_\cS)$ in $\Tv$ is $T^{\Siv}$.
\end{proof}

\section{$L$-groups}\label{sec:L-groups}

If the base field $k$ is not algebraically closed then it is not
really the dual group $\Gv$ itself but its semidirect product
$\tGv=\Gv\rtimes{\rm Gal}(\overline k|k)$ with the Galois group, the
so-called \emph{$L$-group of $G$}, which is of representation
theoretic significance. See, e.g., Borel's introduction \cite{Borel}.

There is evidence that also a spherical variety $X$ should have an
$L$-group $\tGv_X$ attached to it. We won't wager a precise definition
but we are going to give some constraints. In particular, the
existence of equivariant (very) adapted homomorphisms determines the
Galois action on $\Gv_X$ to a large extent.

Let, more generally, $\cS$ be a \wss.. Let moreover $E$ be an abstract
group acting on $\cSq$ and leaving $\cS$ invariant. Clearly, the main
example is furnished by a $G$-variety $X$ which is defined over $k$
and $E$ is the absolute Galois group. Then $E$ acts on the root datum
$\cSq=(\Xiq,S,\Xiqv,\Sv)$ by the so-called $*$-action. Moreover, it
can be shown that then $E$ leaves the \wss. of $X$ invariant
(\cite{KK}*{9.2\,\textit{i)} and paragraphs before 10.5} for details).

We start with a general discussion of $E$-actions on a connected
reductive group $G$ where we assume $E$ to fix $B$ and $T$. Then $E$
will act on the based root system $\cSq=(\Xiq,S,\Xiqv,\Sv)$.
Conversely, assume that an $E$-action on $\cSq$ is given. If
$(e_\a)_{\a\in S}$ is a pinning then this action lifts to a unique
$E$-action on $G$ preserving this pinning, i.e., with
$u(e_\a)=e_{u\a}$ for all $u\in E$. An action of this type will be
called \emph{standard}.

Any two automorphism of $G$ inducing the same automorphism of $\cSq$
differ by an automorphism of the form $\Ad(t)$ where $t$ is a unique
element of the adjoint torus $T\ad:=T/Z(G)$. Thus, isomorphism classes
of $E$-actions on $G$ which are compatible with the $E$-action on
$\cSq$ are parameterized by the cohomology group $H^1(E,T\ad)$. Since
$S$ is a $\ZZ$-basis of $\Xi(T\ad)$, the action of $E$ on $\Xi(T\ad)$
is a permutation representation. Hence
\begin{equation}
  H^1(E,T\ad)=\bigoplus_{\a\in S/E}\Xi(E_\a)
\end{equation}
where $\a\in S$ runs through a set of representatives of the
$E$-orbits. This means that an $E$-action on $G$ is determined by a
system $(\chi_\a)_{\a\in S/E}$ of characters $\chi_\a$ of $E_\a$ in
such a way that
\begin{equation}
  u(e_\a)=\chi_\a(u)e_\a\text{ for $u\in E$ and $\a\in S$ with
    $u\a=\a$.}
\end{equation}
The $E$-action is called \emph{standard in $\a$} if $\chi_\a$ is
trivial. Clearly, the action is standard if and only it is standard in
all $\a$.

The dual group $\Gv$ comes with a pinning which we use to equip it
with a standard $E$-action. Let, moreover, $\cS=(\Xi,\Sigma,S\p)$ be
an $E$-invariant \wss.. Since $E$ stabilizes the associated roots
$\Sa$ as well, the associated group $\Ga_\cS$ is an $E$-stable
subgroup of $\Gv$ and therefore carries an induced $E$-action. This
action is given by a system of characters
$(\chi\ass_\g)_{\gv\in\Sa/E}$ which we are going to determine.

\begin{lemma}

  Let $\gv\in\Sa$ and $u\in E$ with $u\gv=\gv$. Then
  \begin{equation}\label{eq:chiass}
    \chi\ass_\g(u)=
    \begin{cases}
      -1&\text{if $\g\in\Sigma$ is of type $\sA_{2n}\ (n\ge1)$ and
        $\res_{|\g|}u\ne{\rm id}$},\\
      \phantom-1&\text{otherwise.}
    \end{cases}
  \end{equation}

\end{lemma}

\begin{proof}

  Assume first that $u$ acts as identity on $|\gamma|$. Then
  $u(e_{\av})=e_{\av}$ for all $\a\in|\g|$. Since $e_{\gv}$ appears in
  the subalgebra of $\fgv$ generated by $(e_{\av})_{\a\in|\g|}$ we
  also have $u(e_{\gv})=e_{\gv}$ and therefore $\chi\ass_\g(u)=1$.

  Now assume $\res_{|\g|}u\ne{\rm id}$. If $\gv\in\sa$ then $u$ fixes
  $\sigma$ and acts non-trivially on $|\sigma|$. Thus $\sigma$ must be
  of type $\sA_n$ with $n\ge2$ or $\sD_n$ with $n\ge2$. The latter
  possibility is excluded, since otherwise $u$ would not fix $\gamma$.

  Thus, $\sigma$ is of type $\sA_n$, $n\ge2$. We settle this case by
  an explicit computation. The support $|\sigma|$ corresponds to a
  subalgebra of $\fgv$ which is isomorphic to $\sl(N)$ with
  $N=n+1$. For $i\ne j$ let $E_{ij}\in\sl(N)$ be the corresponding
  elementary matrix. One checks that the standard action of $u$ on
  $\sl(N)$ is $u(A)=-JA^tJ^{-1}$ where
  $J={\rm antidiag}(1,-1,1,-1,\ldots)$. This implies
  \begin{equation}
    u(E_{ij})=(-1)^{i+j-1}E_{N+1-j,N+1-i}
  \end{equation}
  Now the root space for $\gamma$ is spanned by $E_{1\,N}$. Thus, the
  assertion follows from $u(E_{1\,N})=(-1)^{n+1}E_{1\,N}$.
\end{proof}

\begin{remark}

  The lemma shows that the $E$-action on $\Ga_\cS$ may be
  non-standard. Nevertheless, this phenomenon seems to be quite
  rare. The tables in \cite{BP} show that if $\cS$ is induced by a
  spherical variety $X=G/H$ with $G$ simple and $H$ reductive then up
  to isogeny there are only two series:
  $X=\SL(2n+1)/S(\GL(m)\GL(2n+1-m))$ with $1\le m\le n$ and
  $X=\SL(2n+1)/\G_m\Sp(2n)$ with $n\ge1$.

\end{remark}

Next, we treat the dual group. There is a slight difference in that
$\Gv_\cS$ is defined abstractly by its root system and not as a
subgroup of $\Gv$. So we have to formulate the result a bit
differently:

\begin{lemma}

  Let $E$ act on $\Gv_\cS$ by means of a system of characters
  $(\chiv_\sigma)_{\sigma\in\Sigma/E}$. Then there exists an adapted
  $E$-equivariant homomorphism $\Gv_\cS\to\Gv$ if and only if for all
  $\sigma\in\Sigma$ and $u\in E$ with $u\sigma=\sigma$:
  \begin{equation}\label{eq:charsys}
    \chiv_\sigma(u)=
    \begin{cases}
      -1&\text{if $\sigma\in\Sigma$ is of type $\sA_{2n}\ (n\ge1)$ and
        $\res_{|\sigma|}u\ne{\rm id}$},\\
      \pm1&\text{if $\sigma\in\Sigma$ is of type $\sD_n\ (n\ge2)$ and
        $\res_{|\sigma|}u\ne{\rm id}$},\\
      \phantom-1&\text{otherwise.}
    \end{cases}
  \end{equation}
\end{lemma}

\begin{proof}

  Because of \eqref{eq:chiass} one can choose a pinning $e_{\gv}$ of
  $\Ga_\cS$ such that $u(e_{\gv})=e_{u\gv}$ whenever $\gamma\in\Sa$ is
  not of type $\sA_{2n\ge2}$. Let $\phi_0:\Gv_\cS\to\Ga_\cS$ be the
  adapted homomorphism which is obtained by folding. This induces a
  pinning $e_\siv$ of $\Gv_\cS$. Observe that $\phi_0$ is
  $E$-equivariant.

  Let $\sigma\in\Sigma$. If $\sigma\in\Phi$ then
  $\phi_0:\fgv_{\siv}\to\fg\ass_{\siv}$ is an isomorphism and we have
  to define $\chiv_\sigma=\chi\ass_\sigma$. Thus assume
  $\sigma=\g_1+\g_2\not\in\Phi$. Then $\phi_0$ maps $\fgv_{\siv}$ into
  $U:=\fg\ass_{\gv_1}\oplus\fg\ass_{\gv_2}$ . If
  $\res_{|\sigma|}u={\rm id}$ then $u$ acts trivially on $U$, forcing
  $\chiv_\sigma(u)=1$. Otherwise, $u$ interchanges the two pinning
  elements $e_{\gv_i}$. Then $U$ contains two $u$-stable
  $1$-dimensional subspaces which are spanned by
  $e_\pm:=e_{\gv_1}\pm e_{\gv_2}$. Since $\phi(\fg_{\siv})$ is one of
  them, we see that $\chiv_{\sigma}=\pm1$. On the other side, clearly
  for any choice of $e_\pm$ there is an $E$-equivariant adapted $\phi$
  with $\phi(e_{\siv})=e_\pm$.
\end{proof}

Next, we study compatibility with the principal $\SL(2)$ of $\LvS$.
The action of $E$ on $\LvS$ is clearly standard. Since the
$W_\cS$-action on $\LvS$ is defined to be standard, we see that $E$
acts on $\LaS$. The fixed point group $L_\cS^{\mathsf L}:=(\LaS)^E$ is
of finite index in the fixed point group $(\LvS)^{\L W_\cS}$ where
$\L W_\cS:=W_\cS\rtimes E$. Note that the principal $\SL(2)$ has
values in $L_\cS^{\mathsf L}$. Recall from \cref{cor:arthur} that for
an adapted $\phi$ to commute with a principal $\SL(2)$ it necessarily
has to be very adapted.

\begin{lemma}

  The following are equivalent:

  \begin{enumerate}

  \item There exists a very adapted $E$-equivariant homomorphism
    $\phi:\Gv_\cS\to\Gv$.

  \item For all $\sigma\in\Sigma$ and $u\in E$ with $u\sigma=\sigma$:
    \begin{equation}\label{eq:charsys2}
      \chiv_\sigma(u)=
      \begin{cases}
        -1&\text{if $\sigma\in\Sigma$ is of type $\sA_{2n}\ (n\ge1)$
          or $\sD_n\ (n\ge3)$ and
          $\res_{|\sigma|}u\ne{\rm id}$},\\
        \pm1&\text{if $\sigma\in\Sigma$ is of type $\sD_2$ and
          $\res_{|\sigma|}u\ne{\rm id}$},\\
        \phantom-1&\text{otherwise.}
      \end{cases}
    \end{equation}
  \end{enumerate}

\end{lemma}

\begin{proof}

  Compatibility with $\psi$ creates no new constraints for
  $\sigma\in\Sigma\cap\Phi$ or for $\sigma$ of type $\sD_2$ since in
  that case $\sigma\ass$ is orthogonal to $S\p$. So let
  $\sigma=\g_1+\g_2$ be of type $\sD_{n\ge3}$ and $u\in E$ with
  $u\sigma=\sigma$ and $\res_{|\sigma|}\ne{\rm id}$. Consider, as in
  the proof of \cref{thm:centralizer}, the subalgebra of $\fgv$ whose
  simple roots are $\dv_1$, $\bv:=\gv_1-\dv_1=\gv_2-\dv_2$, and
  $\dv_2$. It is isomorphic to $\sl(4)$ with basis vectors
  $E_{ij}$. Since $\fgv_{\bv}$ is contained in the subalgebra spanned
  by $\fgv_{\av_1},\ldots,\fgv_{\av_{n-2}}$, the action of $u$ on
  $\fgv_{\bv}$ is trivial. Moreover, the two pinning vectors
  $e_{\dv_1}$ and $e_{\dv_2}$ are interchanged by $u$. This shows that
  the action of $u$ on $\sl(4)$ is standard. The root space
  $(\fgv_\cS)_{\siv}$ is spanned by a vector of the form
  $e:=xE_{13}+yE_{24}$. If $\phi$ is very adapted then $e$ should
  commute with the $u$-invariant vector
  $c:=e_{\dv_1}+e_{\dv_2}=E_{12}+E_{34}$ which forces $x=y$. But then
  we have $u(e)=-e$ proving $\chiv_\sigma(u)=-1$.
\end{proof}

To determine the ``correct'' character $\chiv_\sigma$ when $\sigma$ is
of type $\sD_2$ (if there is any) one needs input from representation
theory. Yiannis Sakellaridis has informed us that the phenomenon of
non-standard $E$-actions is related to the notion of so called
\emph{unstable base change maps}. This connection can be seen as
follows: The $E$-action on $\Gv_\cS$ is determined by an element
$c\in H^1(E,\Av\ad)$ where $\Av\ad=\Av/Z$ and $Z:=Z(\Gv_\cS)$. Let
$\Gv_\cS\rtimes_cE$ be the corresponding semidirect product. If $c$
can be lifted to $\tilde c\in H^1(E,\Av)$ then
$(g,u)\mapsto(\tilde c(u)g\tilde c(u)^{-1},u)$ defines an isomorphism
\begin{equation}
  \Gv_\cS\rtimes_0E\overset\sim\to\Gv_\cS\rtimes_cE
\end{equation}
where the left hand side denotes the semidirect product with respect
to the standard action. The obstruction for the existence of
$\tilde c$ is the image $c_2$ of $c$ in $H^2(E,Z)$. It can be killed
by extending the group $E$, e.g., by replacing it with the central
extension defined by $c_2$. Another possibility is the Weil group:

\begin{lemma}

  Let $k$ be a $p$-adic field. Then $W_k$, its Weil group, acts on
  $\cS$ via its projection to the Galois group of $k$. Assume that $Z$
  is connected. Then
  \begin{equation}\label{eq:Weil}
    \Gv_\cS\rtimes_0W_k\overset\sim\to\Gv_\cS\rtimes_cW_k.
  \end{equation}

\end{lemma}

\begin{proof}

  Indeed, $H^2(W_k,Z)=0$ by \cite{Karpuk}.
\end{proof}

Now the unstable base change map is the composition of \ref{eq:Weil}
with an adapted $W_k$-equivariant homomorphism $\phi$. Thus it is a
homomorphism
\begin{equation}
  \Gv_\cS\rtimes_0W_k\to\Gv\rtimes_0W_k.
\end{equation}

The investigation of distinguished representations for
$X=\GL(2,K)/\GL(2,k)$ with $[K:k]=2$ (see, e.g., Flicker
\cite{Flicker}) indicates that the action of $E$ is non-standard in
this case. Here $S=\{\a,\overline\a\}$ and $\Sigma$ contains a single
root $\a+\overline\a$ which is of type $\sD_2$. If one also assumes
that the $E$-action is compatible with localization, then the correct
action of $E$ on $\Gv_\cS$ would be given by
\begin{equation}
  \chiv_\sigma(u)=
  \begin{cases}
    -1&\text{if $\sigma\in\Sigma$ is of type $\sA_{2n}\ (n\ge1)$ or
      $\sD_n\ (n\ge2)$ and
      $\res_{|\sigma|}u\ne{\rm id}$},\\
    \phantom-1&\text{otherwise.}
  \end{cases}
\end{equation}
There is one more piece of evidence for this which is more in line
with our setup. Consider $G=\SO(2n)$ with $n\ge3$ and
$H=\SO(2n-1)\subset G$. Then $X=G/H$ has one spherical root which is
of type $\sD_n$. Now consider the ($n-2$)-nd maximal parabolic
subgroup $P_{n-2}$ of $\SO(2n-1)$. Then one can show that
$Y=G/P_{n-2}$ is spherical with spherical roots
$\sigma_1=\a_1+\ldots+\a_{n-2}$ and $\tau=\a_{n-1}+\a_n$ (see
\cite{Wasserman}). So $\tau$ is of type $\sD_2$. Because there is a
surjective map $Y\to X$, it is expected that $\Gv_X$ is a subgroup
$\Gv_Y$. Now consider the outer automorphism $u$ of $G$. Then the
action of $u$ on $\Gv_X$ is non-standard. An easy calculation shows
that $\Gv_X$ is an $u$-stable subgroup of $\Gv_Y$ only if the
$u$-action on $\Gv_Y$ is non-standard, as well. Thus $\tau$ should be
non-standard for $u$.

\section{Concluding remarks}

In Section \ref{sec:WSS}, we mentioned a couple of production
procedures for \wsss.. Here we discuss the way they affect dual groups
and centralizers. For this, we fix a \wss. $\cS=(\Xi,\Sigma,S\p)$.

$\bullet$ \emph{Change of $\Xi$:} Let $\Xi_0\subseteq\Xi$ be any
sublattice with $\Sigma\subseteq\Xi_0$. Then
$\cS_0:=(\Xi_0,\Sigma,S\p)$ is another \wss.. There is a canonical
homomorphism $\iota:\Gv_{\cS_0}\to\Gv_\cS$ with finite kernel such
that $\phi_0:=\phi\circ\iota$ is (very) adapted for $\cS_0$ if $\phi$
is (very) adapted for $\cS$. Since $\LaS$ does not depend on $\Xi$ we
get the following diagram
\begin{equation}
  \cxymatrix{\Gv_{\cS_0}\ar[d]\ar[r]&\Gv\ar@{=}[d]&\La_{\cS_0}\ar@{_(->}[l]\ar@{=}[d]\\
    \Gv_\cS\ar[r]&\Gv&\LaS\ar@{_(->}[l]}
\end{equation}
Geometrically, this corresponds to an isogeny $X\to X_0$ of
$G$-varieties.

$\bullet$ \emph{Localization in $\Sigma$:} This case has been
partially dealt with in \cref{thm:unique}. Let
$\Sigma_0\subseteq\Sigma$ be a subset. Then
$\cS_0:=(\Xi,\Sigma_0,S\p)$ is a \wss. (a \emph{boundary degeneration}
in the parlance of \cite{SV}). In this case,
$\Gv_{\cS_0}\subseteq\Gv_\cS$ is a Levi subgroup. Moreover, every
adapted $\phi$ restricts to an adapted $\phi_0$. For the centralizers
we have $\La_{\cS_0}\supseteq \LaS$.
\begin{equation}
  \cxymatrix{\Gv_{\cS_0}\ar@{^(->}[d]\ar[r]&\Gv\ar@{=}[d]&\La_{\cS_0}\ar[l]\\
    \Gv_\cS\ar[r]_\phi&\Gv&\LaS\ar[l]\ar@{_(->}[u]}
\end{equation}
Geometrically, this procedure corresponds to replacing a $G$-variety
by one of its ``boundary components'' in a suitable compactification
(see \cite{SV}).

$\bullet$ \emph{Parabolic induction:} Let $S_0\subseteq S$ be a subset
and let $\cS_0=(\Xi,\Sigma,S\p)$ be a \wss. with respect to (the root
subsystem generated by) $S_0$. Then $\cS=\cS_0$ is a \wss. also with
respect to $S$ and is called a \emph{parabolic induction}. Observe
that, conversely, $\cS$ is parabolically induced from $S_0$ if and
only if
\begin{equation}
  S\p\cup\bigcup_{\sigma\in\Sigma}|\sigma|\subseteq S_0.
\end{equation}
The subset $S_0$ corresponds to a Levi subgroup $\Gv_0\subseteq\Gv$
while dual group and centralizer stay the same:
\begin{equation}
  \cxymatrix{
    \Gv_{\cS_0}\ar[r]^\phi\ar@{=}[d]&\Gv_0\ar@{^(->}[d]&\La_{\cS_0}\ar@{=}[d]\ar@{_(->}[l]\\
    \Gv_\cS\ar[r]&\Gv&\LaS\ar@{_(->}[l]}
\end{equation}
Geometrically, the parabolic induction is the variety $G\times^{P^-}Y$
where $P^-=LU^-$ is a parabolic opposite to $B$ with Levi part $L$ and
$Y$ is an $L$-variety.

$\bullet$ \emph{Removal of compact factors:} Let $S_0\p\subseteq S\p$
with $|\sigma|\cap S\p\subseteq S_0\p$ for all $\sigma\in\Sigma$. Then
$\cS_0=(\Xi,\Sigma,S_0\p)$ is a \wss..
\begin{equation}
  \cxymatrix{
    \Gv_{\cS_0}\ar@{=}[d]\ar[r]&\Gv\ar@{=}[d]&\La_{\cS_0}\ar@{_(->}[d]\ar@{_(->}[l]\\
    \Gv_{\cS}\ar[r]&\Gv&\LaS\ar@{_(->}[l] }
\end{equation}
Observe that this process is, as opposed to the previous ones,
\emph{not} compatible with principal $\SL(2)$'s.  Geometrically, this
process leads to a fibration $X\to X_0$ whose fibers are flag
varieties

$\bullet$ \emph{Localization in $S$:} Let $S_0\subseteq S$ be a
subset. Put $\Sigma_0:=\{\sigma\in\Sigma\mid|\sigma|\subseteq S_0\}$
and $S_0\p:=S_0\cap S\p$. Then $\cS_0=(\Xi,\Sigma_0,S_0\p)$ and
$\cS_1=(\Xi,\Sigma_0,S\p)$ are \wsss. with respect to $S_0$ and $S$,
respectively. This process is the concatenation of the previous three
processes. This turns out to be quite neat on the dual group side but
is slightly messy for centralizers:
\begin{equation}
  \cxymatrix{
    \Gv_{\cS_0}\ar@{^(->}[d]\ar[r]&\Gv_0\ar@{^(->}[d]&\La_{\cS_0}\ar@{_(->}[l]\ar@{^(->}[d]\\
    \Gv_{\cS}\ar[r]&\Gv&\La_{\cS_1}\ar@{_(->}[l]&\LaS\ar@{_(->}[l]}
\end{equation}
Geometrically, localization in $S$ corresponds to looking at a certain
open Bia\l ynicki-Birula cell (see e.g. \cite{KnopLocalization}).

We conclude this paper with a remark on integrality. It is well-known
that, due to its combinatorial construction, the Langlands dual group
$\Gv$ is defined and split over $\ZZ$. Similarly, the dual group
$\Gv_\cS$ and the associated group $\Ga_\cS$ are also defined and
split over $\ZZ$. There is a slight difficulty with the centralizer
$\La_\cS$ due to the appearance of $\SO(3)\subseteq\SL(3)$ which
is only well-behaved outside the prime $2$. Then the following is easy
to verify:

\begin{proposition}\label{thm:integral}

  Let $\cS=(\Xi,\Sigma,S\p)$ be a \wss..

  \begin{enumerate}

  \item The associated subgroup $\Ga_\cS\subseteq\Gv$ is defined over
    $\ZZ$.

  \item There exist adapted homomorphisms $\phi:\Gv_\cS\to\Gv$ which
    are defined over $\ZZ$. Moreover, the group $\Ta\ad(\ZZ)$
    ($\cong\{\pm1\}^r$ with $r=|\Sa|$) acts simply transitively on
    these adapted homomorphisms.

  \item The subgroup $\La_\cS\subseteq\Gv$ is defined and smooth
    over $\ZZ[\half]$.

  \item There exist very adapted homomorphisms $\phi:\Gv_\cS\to\Gv$
    which are defined over $\ZZ[\half]$.

  \end{enumerate}

\end{proposition}

\vfill\newpage

\section{Tables}
\label{sec:appendix}

\begin{table}[h]
$
\begin{array}{llll}

  S:=|\sigma|\cup|\tau|&\sigma,\tau&S\setminus S\p&\\

  \noalign{\smallskip\hrule\medskip\text{Case $\sA$.
      For $\sA_3$ see also $\sD_3$.}\hfill\smallskip}

  \sA_n, n\ge l+1\ge2&\a_1\+\a_l,\a_{l+1}\+\a_n&
  \{\a_1,\a_l,\a_{l+1},\a_n\}&\\

  \sA_n, n\ge4&\8\a_1+\a_n/,\a_2\+\a_{n-1}&
  \{\a_1,\a_2,\a_{n-1},\a_n\}&\\

  \sA_5&\8\a_1+2\a_2+\a_3/,\8\a_3+2\a_4+\a_5/&\{\a_2,\a_4\}&\\

  \sA_2+\sA_2&\8\a_1+\a_1'/,\8\a_2+\a_2'/&S&\\

  \noalign{\smallskip\text{Case $\sB$. For $\sB_2$ see also $\sC_2$.}\hfill\smallskip}

  \sB_n, n\ge p+1\ge2&\a_1\+\a_p,\a_{p+1}\+\a_n&
  \{\a_1,\a_p,\a_{p+1},(\a_n)\}&\\

  \sB_3&\a_1+\a_2,\a_2+\a_3&S&\\

  \sB_4&\a_1+\a_2+\a_3+\a_4,\8\a_2+2\a_3+3\a_4/&\{\a_1,\a_4\}&**\\

  \noalign{\smallskip\text{Case $\sC$.}\hfill\smallskip}

  \sC_n,n\ge2&\a_1,\a_1+2\a_2\+2\a_{n-1}+\a_n&\{\a_1,\a_2\}&*\\

  \sC_n,n\ge3&\9\a_1+\a_2,\a_2+2\a_3\+2\a_{n-1}+\a_n&\{\a_1,\a_2,\a_3\}&\\

  \sC_n,n\ge4&\9\8\a_1+2\a_2+\a_3/,\a_3+2\a_4\+2\a_{n-1}+\a_n&\{\a_2,\a_4\}&\\

  \sC_n,n\ge p+2\ge3&\9\a_1\+\a_p,
  \a_{p+1}+2\a_{p+2}\+2\a_{n-1}+\a_n&
  \{\a_1,\a_p,\a_{p+1},\a_{p+2}\}&\\

  \sC_n+\sA_1,n\ge2&\9\8\a_1+\a_1'/,\a_1+2\a_2\+2\a_{n-1}+\a_n&
  \{\a_1,\a_2,\a_1'\}&*\\

  \sC_n,n\ge2&\a_1\+\a_{n-1},\a_n&\{\a_1,\a_{n-1},\a_n\}&\\

  \sC_n,n\ge3&\8\a_1+\a_n/,\a_2\+\a_{n-1}&\{\a_1,\a_2,\a_{n-1},\a_n\}&\\

  \sC_2+\sC_2&\8\a_1+\a_1'/,\8\a_2+\a_2'/&S&\\

  \noalign{\smallskip\text{Case $\sD$. For $\sD_3$ see also $\sA_3$.}\hfill\smallskip}

  \sD_n,n\ge p+3\ge4\quad&\9\a_1\+\a_p,\82\a_{p+1}\+2\a_{n-2}+\a_{n-1}+\a_n/&
  \{\a_1,\a_p,\a_{p+1}\}&\\

  \sD_n,n\ge3&\a_1\+\a_{n-2},\8\a_{n-1}+\a_n/&\{\a_1,\a_{n-2},\a_{n-1},\a_n\}&\\

  \sD_5&\8\a_1+2\a_2+\a_3/,\a_3+\a_4+\a_5&\{\a_2,\a_4,\a_5\}&\\

  \sD_n,n\ge3&\9\a_1\+\a_{n-2}+\a_{n-1},\a_1\+\a_{n-2}+\a_n&
  \{\a_1,\a_{n-1},\a_n\}&\\

  \noalign{\smallskip\text{Case $\sE\sF\sG$.}\hfill\smallskip}

  \sE_6&\9\82\a_1+2\a_3+2\a_4+\a_2+\a_5/,\82\a_6+2\a_5+2\a_4+\a_3+\a_2/&
  \{\a_1,\a_6\}&\\

  \sE_6&\9\a_1+\a_3+\a_4+\a_5+\a_6,\82\a_2+2\a_4+\a_3+\a_5/&
  \{\a_1,\a_2,\a_6\}&\\

  \sF_4&\a_1+\a_2,\a_3+\a_4&S&\\

  \sF_4&\a_1+\a_2+\a_3,\a_4&\{\a_1,\a_3,\a_4\}&\\

  \sF_4&\a_1+\a_2+\a_3,\a_4+2\a_3+\a_2&\{\a_1,\a_3,\a_4\}&\\

  \sF_4&\8\a_1+\a_4/,\a_2+\a_3&S&\\

  \sF_4&\8\a_1+2\a_2+3\a_3/,\a_4&\{\a_3,\a_4\}&\\

  \sG_2&\a_1,\a_2&S&\\

  \sG_2&\a_1,\a_1+\a_2&S&*\\

  \sG_2+\sG_2&\8\a_1+\a_1'/,\8\a_2+\a_2'/&S&\\

\end{array}
$
\medskip
\caption{\Wsss. of rank 2}\label{tab:rank2}
\end{table}

\begin{landscape}

\addtolength{\topmargin}{10mm}
\addtolength{\evensidemargin}{-7mm}

\begin{table}[h]
$\begin{array}{|ll|l|ll|ll|l|}

    \noalign{\hrule}

    \fg&\fh&\fgv&\fg\ass_X&\fgv_X&\fl_X^\vee&\fl\ass_X&\\

    \noalign{\hrule}

    \sl(n)&\fs(\gl(m){+}\gl(n{-}m))&\sl(n)&\fs(\gl(2m)+\ft^{n-2m})&\sp(2m)&\fs(\gl(n{-}2m){+}\ft^{2m})&\gl(n-2m)&n\ge 2m\ge0\\

    \sl(n)&\so(n)&\sl(n)&\sl(n)&\sl(n)&\ft^{n-1}&0&n\ge1\\

    \sl(2n)&\sp(2n)&\sl(2n)&\fs(\gl(n)+\gl(n))&\sl(n)&\fs(\gl(2)^n)&\sl(2)&n\ge1\\

    \sl(2n+1)&\sp(2n)+\ft^1&\sl(2n+1)&\fs(\gl(n{+}1){+}\gl(n))&\sl(n+1)+\sl(n)&\ft^{2n}&\ft^1&n\ge1\\

    \sl(2n+1)&\sp(2n)&\sl(2n+1)&\fs(\gl(n{+}1){+}\gl(n))&\fs(\gl(n{+}1){+}\gl(n))&\ft^{2n}&\ft^1&n\ge1\\

    \noalign{\hrule}

    \so(2n+1)&\so(m){+}\so(2n{+}1{-}m)&\sp(2n)&\sp(2m)+\ft^{n-m}&\sp(2m)&\sp(2n-2m)+\ft^m&\sp(2n-2m)&n\ge m\ge0\\

    \so(4n+1)&\gl(2n)&\sp(4n)&\sp(2n)+\sp(2n)&\sp(2n)+\sp(2n)&\ft^{2n}&0&n\ge0\\

    \so(4n+3)&\gl(2n+1)&\sp(4n+2)&\sp(2n{+}2){+}\sp(2n)&\sp(2n{+}2){+}\sp(2n)&\ft^{2n+1}&0&n\ge0\\

    \so(7)&\sG_2&\sp(6)&\sl(2)^2+\ft^1&\sl(2)&\gl(3)&\so(3)&\\

    \so(9)&\mathfrak{spin}(7)&\sp(8)&\sl(2)^3+\ft^1&\sl(2)^2&\gl(3)+\ft^1&\so(3)&\\

    \noalign{\hrule}

    \sp(2n)&\sp(2m){+}\sp(2n{-}2m)&\so(2n+1)&\gl(2m)+\ft^{n-2m}&\sp(2m)&\gl(2)^m{+}\so(2n{-}4m{+}1)&\sl(2){+}\so(2n{-}4m{+}1)&n\ge 2m\ge2\\

    \sp(2n)&\gl(n)&\so(2n+1)&\so(2n+1)&\so(2n+1)&\ft^n&0&n\ge0\\

    \sp(2n)&\sp(2n-2)+\ft^1&\so(2n+1)&\so(4)+\ft^{n-2}&\so(4)&\so(2n-3)+\ft^2&\so(2n-3)&n\ge2\\

    \noalign{\hrule}

    \so(2n)&\so(m)+\so(2n-m)&\so(2n)&\so(2m{+}2){+}\ft^{n{-}m{-}1}&\so(2m+1)&\so(2n-2m)+\ft^m&\so(2n-2m-1)&n>m\ge1\\

    \so(2n)&\so(n)+\so(n)&\so(2n)&\so(2n)&\so(2n)&\ft^n&0&n\ge2\\

    \so(4n)&\gl(2n)&\so(4n)&\gl(2n)&\sp(2n)&\gl(2)^n&\sl(2)&n\ge1\\

    \so(4n+2)&\gl(2n+1)&\so(4n+2)&\gl(2n)+\so(2)&\sp(2n)&\gl(2)^n+\so(2)&\sl(2)+\so(2)&n\ge1\\

    \so(8)&\sG_2&\so(8)&\gl(2)+\so(4)&\sl(2)+\so(4)&\sl(2)+\ft^3&\sl(2)&\\

    \so(10)&\mathfrak{spin}(7)+\ft^1&\so(10)&\gl(2)+\so(6)&\sl(2)+\so(6)&\sl(2)+\ft^4&\sl(2)&\\

    \noalign{\hrule}

    \sE_6&\sF_4&\sE_6&\gl(3)+\gl(3)&\sl(3)&\so(8)+\ft^2&\sG_2&\\

    \sE_6&\so(10)+\ft^1&\sE_6&\sl(4)+\ft^3&\sp(4)&\sl(4)+\ft^3&\sp(4)+\ft^1&\\

    \sE_6&\so(10)&\sE_6&\sl(4)+\ft^3&\sp(4)+\ft^1&\sl(4)+\ft^3&\sp(4)+\ft^1&\\

    \sE_6&\sl(6)+\sl(2)&\sE_6&\sE_6&\sF_4&\ft^6&0&\\

    \sE_6&\sp(8)&\sE_6&\sE_6&\sE_6&\ft^6&0&\\

    \sE_7&\sE_6+\ft^1&\sE_7&\sl(6)+\ft^2&\sp(6)&\so(8)+\ft^3&\sG_2&\\

    \sE_7&\so(12)+\sl(2)&\sE_7&\sE_6+\ft^1&\sF_4&\sl(2)^3+\ft^4&\sl(2)&\\

    \sE_7&\sl(8)&\sE_7&\sE_7&\sE_7&\ft^7&0&\\

    \sE_8&\sE_7+\sl(2)&\sE_8&\sE_6+\ft^2&\sF_4&\so(8)+\ft^4&\sG_2&\\

    \sE_8&\so(16)&\sE_8&\sE_8&\sE_8&\ft^8&0&\\

    \sF_4&\so(9)&\sF_4&\sl(2)+\ft^3&\sl(2)&\sp(6)+\ft^1&\sp(6)&\\

    \sF_4&\sp(6)+\sl(2)&\sF_4&\sF_4&\sF_4&\ft^4&0&\\

    \sG_2&\sl(3)&\sG_2&\gl(2)&\sl(2)&\gl(2)&\sl(2)&\\

    \sG_2&\sl(2)+\sl(2)&\sG_2&\sG_2&\sG_2&\ft^2&0&\\

    \noalign{\hrule}

  \end{array}$
\medskip
\caption{Lie algebras of dual groups for $X=G/H$ spherical, $G$ simple, $H$ reductive (see \cites{Kraemer,BP}).}
\label{tab:Kraemer}
\end{table}
\end{landscape}


\newpage

\begin{bibdiv}
  \begin{biblist}

\bib{Akhiezer}{article}{
  author={Ahiezer, Dmitry},
  title={Equivariant completions of homogeneous algebraic varieties by homogeneous divisors},
  journal={Ann. Global Anal. Geom.},
  volume={1},
  date={1983},
  pages={49--78},
}

\bib{Bou}{book}{
  author={Bourbaki, N.},
  title={Éléments de mathématique. Fasc. XXXIV. Groupes et algèbres de Lie. Chapitre IV: Groupes de Coxeter et systèmes de Tits. Chapitre V: Groupes engendrés par des réflexions. Chapitre VI: systèmes de racines},
  series={Actualités Scientifiques et Industrielles, No. 1337},
  publisher={Hermann, Paris},
  date={1968},
  pages={288 pp. (loose errata)},
}

\bib{Borel}{article}{
  author={Borel, Armand},
  title={Automorphic $L$-functions},
  conference={ title={Automorphic forms, representations and $L$-functions}, address={Proc. Sympos. Pure Math., Oregon State Univ., Corvallis, Ore.}, date={1977}, },
  book={ series={Proc. Sympos. Pure Math., XXXIII}, publisher={Amer. Math. Soc., Providence, R.I.}, },
  date={1979},
  pages={27--61},
}

\bib{BdS}{article}{
  author={Borel, Armand},
  author={De Siebenthal, Jean},
  title={Les sous-groupes fermés de rang maximum des groupes de Lie clos},
  journal={Comment. Math. Helv.},
  volume={23},
  date={1949},
  pages={200--221},
}

\bib{Bravi}{article}{
  author={Bravi, Paolo},
  title={Primitive spherical systems},
  journal={Trans. Amer. Math. Soc.},
  volume={365},
  date={2013},
  pages={361--407},
  arxiv={0909.3765},
}

\bib{BP}{article}{
  author={Bravi, P.},
  author={Pezzini, G.},
  title={The spherical systems of the wonderful reductive subgroups},
  journal={J. Lie Theory},
  volume={25},
  date={2015},
  pages={105--123},
  arxiv={1109.6777},
 }

\bib{BraviPezzini}{article}{
  author={Bravi, Paolo},
  author={Pezzini, Guido},
  title={Primitive wonderful varieties},
  journal={Math. Z.},
  volume={282},
  date={2016},
  pages={1067--1096},
  arxiv={1106.3187},
}

\bib{Brion}{article}{
  author={Brion, Michel},
  title={Vers une généralisation des espaces symétriques},
  journal={J. Algebra},
  volume={134},
  date={1990},
  pages={115--143},
}

\bib{BrionPauer}{article}{
  author={Brion, Michel},
  author={Pauer, Franz},
  title={Valuations des espaces homogènes sphériques},
  journal={Comment. Math. Helv.},
  volume={62},
  date={1987},
  pages={265--285},
}

\bib{Flicker}{article}{
  author={Flicker, Yuval Z.},
  title={On distinguished representations},
  journal={J. Reine Angew. Math.},
  volume={418},
  date={1991},
  pages={139--172},
}

\bib{GaitsgoryNadler}{article}{
  author={Gaitsgory, Dennis},
  author={Nadler, David},
  title={Spherical varieties and Langlands duality},
  journal={Mosc. Math. J.},
  volume={10},
  date={2010},
  pages={65--137, 271},
  arxiv={math/0611323},
}

\bib{Kac}{book}{
  author={Kac, Victor G.},
  title={Infinite-dimensional Lie algebras},
  edition={3},
  publisher={Cambridge University Press, Cambridge},
  date={1990},
  pages={xxii+400},
}

\bib{Karpuk}{article}{
  author={Karpuk, David A.},
  title={Cohomology of the Weil group of a $p$-adic field},
  journal={J. Number Theory},
  volume={133},
  date={2013},
  pages={1270--1288},
}

\bib{KnopIB}{article}{
  author={Knop, Friedrich},
  title={Über Bewertungen, welche unter einer reduktiven Gruppe invariant sind},
  journal={Math. Ann.},
  volume={295},
  date={1993},
  pages={333--363},
}

\bib{KnopAB}{article}{
  author={Knop, Friedrich},
  title={The asymptotic behavior of invariant collective motion},
  journal={Invent. Math.},
  volume={116},
  date={1994},
  pages={309--328},
}

\bib{KnopHC}{article}{
  author={Knop, Friedrich},
  title={A Harish-Chandra homomorphism for reductive group actions},
  journal={Ann. of Math. (2)},
  volume={140},
  date={1994},
  pages={253--288},
}

\bib{KnopAuto}{article}{
  author={Knop, Friedrich},
  title={Automorphisms, root systems, and compactifications of homogeneous varieties},
  journal={J. Amer. Math. Soc.},
  volume={9},
  date={1996},
  pages={153--174},
}

\bib{KnopLocalization}{article}{
  author={Knop, Friedrich},
  title={Localization of spherical varieties},
  journal={Algebra Number Theory},
  volume={8},
  date={2014},
  pages={703--728},
  arxiv={1303.2561},
}

\bib{KnopSRSV}{article}{
  author={Knop, Friedrich},
  title={Spherical roots of spherical varieties},
  journal={Ann. Inst. Fourier (Grenoble)},
  volume={64},
  date={2014},
  pages={2503--2526},
  arxiv={1303.2466},
}

\bib{KK}{article}{
  author={Knop, Friedrich},
  author={Krötz, Bernhard},
  title={Reductive group actions},
  journal={Preprint},
  date={2016},
  pages={62 pp.},
  arxiv={1604.01005},
}

\bib{Kraemer}{article}{
  author={Krämer, Manfred},
  title={Sphärische Untergruppen in kompakten zusammenhängenden Liegruppen},
  journal={Compositio Math.},
  volume={38},
  date={1979},
  pages={129--153},
}

\bib{Luna}{article}{
  author={Luna, Domingo},
  title={Variétés sphériques de type $A$},
  journal={Publ. Math. Inst. Hautes Études Sci.},
  number={94},
  date={2001},
  pages={161--226},
}

\bib{Res}{article}{
  author={Ressayre, Nicolas},
  title={Spherical homogeneous spaces of minimal rank},
  journal={Adv. Math.},
  volume={224},
  date={2010},
  pages={1784--1800},
  arxiv={0909.0653},
}

\bib{SV}{article}{
  author={Sakellaridis, Yiannis},
  author={Venkatesh, Akshay},
  title={Periods and harmonic analysis on spherical varieties},
  date={2012},
  pages={291p.},
  arxiv={1203.0039},
}

\bib{Schalke}{thesis}{
  author={Schalke, Barbara},
  title={Beiträge zur Theorie sphärischer Varietäten},
  subtitle={Die duale Gruppe und Betrachtungen in positiver Charakteristik},
  type={Dissertation},
  organization={FAU Erlangen-Nürnberg},
  language={english},
  date={2016},
  pages={94pp.},
}

\bib{Springer}{book}{
  author={Springer, Tonny A.},
  title={Linear algebraic groups},
  series={Progress in Mathematics},
  volume={9},
  edition={2},
  publisher={Birkhäuser Boston, Inc., Boston, MA},
  date={1998},
  pages={xiv+334},
}

\bib{Wasserman}{article}{
  author={Wasserman, Benjamin},
  title={Wonderful varieties of rank two},
  journal={Transform. Groups},
  volume={1},
  date={1996},
  pages={375--403},
}


  \end{biblist}
\end{bibdiv}
\end{document}